\documentclass[12pt,a4paper]{amsart}
\usepackage{geometry}
\usepackage{amsmath,amssymb,amsfonts,amscd,amsthm,wasysym,color,enumitem,indentfirst,graphicx,booktabs}
\usepackage[bookmarksnumbered,colorlinks,linktocpage,plainpages]{hyperref}
\hypersetup{pdfstartview={FitH}}
\usepackage{mathtools}
\usepackage{mathrsfs}

\numberwithin{equation}{section}
\newtheorem{theorem}{Theorem}[section]

\newtheorem{lemma}[theorem]{Lemma}

\theoremstyle{definition}
\newtheorem{definition}[theorem]{Definition}

\newtheorem{remark}[theorem]{Remark}

\usepackage{tikz}

\newtheorem{conjecture}{Conjecture}

\allowdisplaybreaks[4]

\makeatletter
\@namedef{subjclassname@2020}{\textup{2020} Mathematics Subject Classification}
\makeatother

\def\abs#1{\left|#1\right|}
\date{}

\def\R{\mathbb{R}}

\def\O{\mathbb{O}}

\newcommand{\bfe} {{\mathbf e}}
\begin{document}
\title[Gilbert's conjecture ]{Resolving Gilbert's Conjecture: Dimensional Dependencies in Hardy Spaces Valued in Clifford Modules}
\date{\today}

\author{Yong Li}
\address{School of Mathematics and Statistics, Anhui Normal University, Wuhu 241002, Anhui, People's Republic of China.}
\email{leeey@ahnu.edu.cn}

\author{Guangbin Ren}
\address{%
 School of Mathematical Sciences, University of Science and Technology of China,Hefei, Anhui 230026, People's Republic of China.}
\email{rengb@ustc.edu.cn}

\thanks{Yong Li is supported by University Annual Scientific Research Plan of Anhui Province(2022AH050175).
Guangbin Ren is supported  by the National Natural Science Foundation of China (Grant Nos.12171448)}


	\begin{abstract}
	This article provides a thorough investigation into Gilbert's Conjecture, pertaining to Hardy spaces in the upper half-space valued in Clifford modules. We explore  the conjecture proposed by Gilbert in 1991, which seeks to extend the classical principle of representing real
$L^p$
functions on the real line as boundary values of Hardy holomorphic functions to higher-dimensional Euclidean spaces valued in any Clifford module. We present a complete  resolution to this conjecture, demonstrating that its validity is contingent upon the dimension $n$,
specifically holding true when \(n \not\equiv 6, 7 \mod 8\)
and failing otherwise.
The pivotal discovery that Gilbert's conjecture can be reformulated as a set of algebraic conditions is underscored in this work. To navigate these conditions, we employ a novel strategy that leverages the octonions, revealing their instrumental role in addressing issues related to Clifford modules and spinors. This innovative approach not only provides explicit realization through the generalization of the Hilbert transform to the Riesz transform but also establishes a significant advancement in the understanding of Hardy spaces within higher dimensions.	\end{abstract}

\maketitle
\tableofcontents
\section{Introduction}
 	Hardy space theory elegantly bridges real \(L^p\) spaces with analytic domains, establishing an isomorphism
	\[H^p(\mathbb{C}_+, \mathbb{C}) = L^p(\mathbb{R}, \mathbb{R})\]
	for \(p>1\), with \(\mathbb{C}_+\) representing the upper half-plane. This isomorphism implies that for any real function \(f \in L^p(\mathbb{R}, \mathbb{R})\), there exists a complex counterpart whose imaginary part is derived from the Hilbert transform, manifesting as the non-tangential boundary limit of a holomorphic function in \(H^p(\mathbb{C}_+, \mathbb{C})\). The Hardy holomorphic extension of any complex signal is, in turn, characterized by the Cauchy integral.

	When extending Hardy space theory into higher dimensions, one is led to conjecture an isomorphism
\[H^p(\mathbb{R}^{n}_+, C\ell_{n}) \cong L^p(\mathbb{R}^{n-1}, C\ell_{n-1}),\]
with the homomorphic function theory expanded to \(\mathbb{R}^{n}\) via Clifford analysis
\cite{BDS82,Gilbert,F.S, DGK91,GM88}.
 This expansion utilizes the Dirac operator
\[\sum_{j=0}^{n-1}\mathbf{e}_j\frac{\partial }{\partial x_j},\]
where \(\mathbf{e}_0, \mathbf{e}_1, \ldots, \mathbf{e}_{n-1}\) form the canonical basis of \(\mathbb{R}^{n}\).

More generally, 	by considering any Clifford module \(\mathfrak{H}\) and its subspace $\mathfrak{H}_0\) with the decomposition $$\mathfrak{H} = \mathfrak{H}_0 \oplus \eta \mathfrak{H}_0$$ for some \(\eta \in \mathbb{R}^n\) where \(\eta^2 = -1\), Gilbert conjectured the isomorphism
	\[H^p(\mathbb{R}^{n}_+, \mathfrak{H}) \cong L^p(\mathbb{R}^{n-1}, \mathfrak{H}_0).\]
	
	This paper aims to resolve Gilbert's conjecture, revealing a surprising dependency on the dimensionality of the space: the conjecture holds when \(n \not\equiv 6, 7 \mod 8\) and fails otherwise.
	
	We provide an explicit realization of this isomorphism through the generalization of the Hilbert transform to the Riesz transform. Further details are explored in the subsequent sections.

\subsection{Real \(H^p\)-Theory for Complex Variables}\label{op-boun-Cauchy}
Hardy spaces
 play a fundamental role in various areas including complex analysis, harmonic analysis, operator theory, complex dynamics,  analytic number theory,  and partial differential equations.   The development and study of Hardy spaces have deeply influenced the understanding of function theory and have found applications in signal processing and  control theory.

In 1915,   Hardy \cite{Hardy15} laid the groundwork for classical \(H^p\) theory, examining spaces of functions that are analytic in the upper half-plane \(\mathbb{C}_+\), characterized by the norm
\[\|F\|_{H^p(\mathbb{C}_+)} = \sup_{t > 0} \left(\int_{\mathbb{R}} |F(x + \mathrm{i}t)|^p \, \mathrm{d}x\right)^{1/p} < \infty.\]

 Hardy's original motivation stemmed from his interest in the boundary behavior of analytic functions.
  Research on Hardy spaces has always been a hot topic in mathematical studies \cite{SW60,FS72,QR96,BDQ18,GTM226,BKN18}.

A fundamental aspect of \(H^p\) theory is that functions within this space can be uniquely identified by the real part of their boundary values on the non-tangential (n.t.) approach. This property is crucial for applications in signal analysis where signals are represented as functions in \(L^p(\mathbb{R}, \mathbb{R})\).

Central to this theory is the correlation between a holomorphic function and its boundary values, particularly determined through n.t. limits. This is formalized in the following theorem:

\begin{theorem}
	Let \(F \in H^p(\mathbb{C}_+), p > 1\). Then, there exists a function \(f \in L^p(\mathbb{R}, \mathbb{C})\) satisfying:
	\begin{enumerate}
		\item \(\lim \limits_{z \to x \, \text{n.t.}} F(z) = f(x)\) for almost every \(x \in \mathbb{R}\),
		\item \(\lim \limits_{t \to 0} \int_{-\infty}^{+\infty} |F(x + \mathrm{i}t) - f(x)|^p \, \mathrm{d}x = 0.\)
	\end{enumerate}
\end{theorem}

This result motivates the definition of the boundary operator $$\mathcal{B}: H^p(\mathbb{C}_+) \to L^p(\mathbb{R}, \mathbb{C}),$$  mapping each Hardy function \(F \in H^p(\mathbb{C}_+)\) to its boundary function \(f = \mathcal{B}(F)\).

Furthermore, classical Hardy space theory establishes that the Cauchy integral operator \(\mathcal{C}\) acts as the left inverse to \(\mathcal{B}\), i.e.,
\[\mathcal{C}\mathcal{B} = Id_{H^p(\mathbb{C}_+)}.\]
The Cauchy integral operator is defined as
\[\mathcal{C}(f)(z) = \frac{1}{2\pi \mathrm{i}} \int_{-\infty}^{+\infty} \frac{f(t)}{t - z} \, \mathrm{d}t, \quad z \in \mathbb{C} \setminus \mathbb{R}.\]

However, \(\mathcal{C}\) does not serve as the true inverse of \(\mathcal{B}\). Utilizing the Hilbert transform \(H\), we find
\[\mathcal{B}\mathcal{C}(f) = \frac{1}{2}(f + \mathrm{i}Hf),\]
where \(  H\) is defined as
\[ H(f)(x) = \text{p.v.}\frac{1}{\pi} \int_{-\infty}^{+\infty} \frac{f(y)}{x - y} \, \mathrm{d}y.\]

Since
$$ f=\mathcal B F=\mathcal B \mathcal C\mathcal B F =\mathcal B \mathcal C f= \frac{1}{2}(f + \mathrm{i} Hf),$$
   the effective range of the boundary operator \(\mathcal{B}\) is captured by
\[\mathcal{B}(H^p(\mathbb{C}_+)) = \{f \in L^p(\mathbb{R}, \mathbb{C}) : f = \mathrm{i} Hf\}. \]
  We define this particular subset of functions as \(\mathbf{R}H^p(\mathbb{C}_+)\), referred to as the Real \(H^p\) space. Within this framework, the Cauchy integral operator \(\mathcal{C}\), when applied to \(\mathbf{R}H^p(\mathbb{C}_+)\), acts as the precise inverse of \(\mathcal{B}\). This leads to the isomorphism
\begin{equation}\label{eq:RealHp}
  H^p(\mathbb{C}_+) \cong \mathbf{R}H^p(\mathbb{C}_+),
\end{equation}
for \(p > 1\), effectively bridging the Hardy spaces over the complex upper half-plane with their real-valued counterparts on the boundary.

Now we consider the operator
\[\mathcal{R}: L^p(\mathbb{R}, \mathbb{C}) \to L^p(\mathbb{R}, \mathbb{R}), \quad \mathcal{R}(f) = 2\,\text{Re}(f).\]
It  ensures that the sequence of mappings
\[L^p(\mathbb{R}, \mathbb{R}) \hookrightarrow L^p(\mathbb{R}, \mathbb{C}) \xrightarrow{\mathcal{C}} H^p(\mathbb{C}_+) \xrightarrow{\mathcal{B}} L^p(\mathbb{R}, \mathbb{C}) \xrightarrow{\mathcal{R}} L^p(\mathbb{R}, \mathbb{R})\]
results in the identity operation on \(L^p(\mathbb{R}, \mathbb{R})\). Similarly,
\[H^p(\mathbb{C}_+) \xrightarrow{\mathcal{B}} L^p(\mathbb{R}, \mathbb{C}) \xrightarrow{\mathcal{R}} L^p(\mathbb{R}, \mathbb{R}) \hookrightarrow L^p(\mathbb{R}, \mathbb{C}) \xrightarrow{\mathcal{C}} H^p(\mathbb{C}_+)\]
establishes the identity on \(H^p(\mathbb{C}_+)\). Consequently, we define
\[\mathbf{R}H^p(\mathbb{R}) = \mathcal{R} \mathcal{B} H^p(\mathbb{C}_+)\]
as the collection of real functions \(f\), where \(f = \text{Reg}\), and \(g = \mathcal{B}(F)\) for some \(F \in H^p(\mathbb{C}_+)\), corresponding to the real parts of the non-tangential limits from the Hardy space. Thus, we establish the isomorphism
\begin{equation}\label{eq:RealHpR}
  H^p(\mathbb{C}_+) \cong \mathbf{R}H^p(\mathbb{R}) = L^p(\mathbb{R}, \mathbb{R}),
\end{equation}
for \(p > 1\), highlighting the intrinsic connection between Hardy spaces on the upper half-plane and real-valued \(L^p\) functions.

\subsection{Hardy Spaces in Harmonic Analysis}
A major advancement in the theory of Hardy spaces came with the work of  Stein and  Weiss in the 1960s \cite{SW60}, who extended the concept of Hardy spaces to high  dimensions and developed tools for their analysis.

In the expanded framework of Hardy theory within the upper half-space $\mathbb{R}^{n+1}_+$, the concept is articulated through the Riesz system \cite{SW71, FS72, S93, Krantz23}. The Riesz system is delineated for a vector-valued function $F=(u_0,u_1, u_2, \ldots, u_n)$ by a set of equations:
\begin{align*}
&	\sum_{j=0}^{n}\frac{\partial u_j}{\partial x_j}= 0, \\
&	\frac{\partial u_i}{\partial x_j} = \frac{\partial u_j}{\partial x_i}
\end{align*}
for every $i,j = 0,1,2,\ldots,n$. A function $F$ is
is said to belong to the Hardy space     if it meets the criteria of the Riesz system and fulfills the condition:
\[
\sup_{y>0}\int_{\mathbb{R}^n}\left|F(x,y)\right|^p \mathrm{d}x   < +\infty.
\]

The Riesz system can be effectively understood within the framework of Clifford algebras. When viewing $F=(u_0,u_1, u_2, \ldots, u_n)$ as a Clifford-valued function $f=\sum_{i=0}^{n}u_i\mathbf{e}_i$, with its domain in $\mathbb{R}^{n+1}$, $f$ adheres to the Riesz system upon being nullified by the Dirac operator:
\[
\sum_{j=0}^{n}\mathbf{e}_j\frac{\partial}{\partial x_j}f = \sum_{j=0}^{n}\mathbf{e}_j\frac{\partial}{\partial x_j}\sum_{i=0}^{n}u_i\mathbf{e}_i = \sum_{j=0}^{n}\frac{\partial u_j}{\partial x_j} + \sum_{0\leqslant i<j\leqslant n}\left(\frac{\partial u_i}{\partial x_j}-\frac{\partial u_j}{\partial x_i}\right)\bfe_j\bfe_i = 0.
\]
Inversely, if $F$ is a Clifford-valued function within $\mathbb{R}^{n+1}$ and is extinguished by the Dirac operator, then $F$ complies with the Riesz system.

Moreover, the Hardy space correspondent to the Riesz system is characterized as the space of vector-valued functions that are extinguished by the Dirac operator, thereby suggesting that Clifford analysis provides a coherent framework for the extension of $H^p$ theory into high  dimensions.

\subsection{Hardy Spaces in Clifford Analysis}

The development of Clifford analysis has significantly enriched the theory of Hardy spaces structured over Clifford modules \cite{QR96,GM88,Gilbert}. This advancement incorporates the consideration of operators within the context of Clifford algebras.

Given a Clifford module $\mathfrak{H}$ (as outlined in Definition \ref{def:CM}), a function $F$ is said to reside in the Hardy space $H^p(\mathbb{R}^n_+, \mathfrak{H})$ if it is Clifford-analytic in the upper half-space
$$\mathbb{R}^n_+ = \{(t, \underline{x}): t > 0, \underline{x} \in \mathbb{R}^{n-1}\}$$
and satisfies the norm condition
$$\|F\|_{H^p(\mathbb{R}^n_+, \mathfrak{H})} = \sup_{t > 0} \left(\int_{\mathbb{R}^{n-1}} |F(t, \underline{x})|^p \, \mathrm{d}x\right)^{1/p} < \infty.$$

To elucidate the relationship between Clifford-analytic functions and their boundary values for $p > \frac{n-2}{n-1}$, we establish that each function in $H^p(\mathbb{R}^n_+, \mathfrak{H})$ possesses non-tangential limits almost everywhere through the following theorem:

\begin{theorem}\cite[P.120 Theorem 5.4]{Gilbert}\label{thm:bdv}
	Let $F \in H^p(\mathbb{R}^n_+, \mathfrak{H}), p > \frac{n-2}{n-1}$. There exists a function $f \in L^p(\mathbb{R}^{n-1}, \mathfrak{H})$ such that:
	\begin{enumerate}
		\item $\lim\limits_{z \to x \, \text{n.t.}} F(z) = f(x)$ for almost all $x \in \mathbb{R}^{n-1}$,
		\item $\lim\limits_{t \to 0} \int_{\mathbb{R}^{n-1}} |F(t, \underline{x}) - f(x)|^p \, \mathrm{d}x = 0$.
	\end{enumerate}
\end{theorem}

As discussed in Subsection \ref{op-boun-Cauchy} for $n=2$, we introduced two operators: the boundary operator and the Cauchy operator. This notation extends to the general case of $n$, albeit with a level of ambiguity.

More precisely, we define two operators:

(1) The boundary function operator:
\begin{equation}
	\label{eq:opB}
	\mathcal{B} : H^p(\mathbb{R}^n_+, \mathfrak{H}) \rightarrow L^p(\mathbb{R}^{n-1}, \mathfrak{H})
\end{equation}
maps $F \in H^p(\mathbb{R}^n_+, \mathfrak{H})$ to its non-tangential limit $f = \mathcal{B}(F) \in L^p(\mathbb{R}^{n-1}, \mathfrak{H})$.

(2) The Cauchy integral operator, pivotal in extending Clifford-analytic functions from their boundary values in higher-dimensional Hardy spaces, is defined for $f \in L^p(\mathbb{R}^{n-1}, \mathfrak{H}), p \geqslant 1,$ by
\begin{equation}
	\label{eq:Cauchy}
	\mathcal{C}f(z) = \frac{1}{\omega_n} \int_{\mathbb{R}^{n-1}} \frac{u - z}{|u - z|^n} \mathbf{e}_0 f(u) \, \mathrm{d}u
\end{equation}
for any $z = (t, \underline{x}) \in \mathbb{R}^n_+$, where $\omega_n$ denotes the surface measure of the unit sphere in $\mathbb{R}^n$.

The Cauchy integral operator has a close relationship with the Riesz transform, the Poisson kernel, and the conjugate Poisson kernel. The $j$th Riesz transform, $ {R}_j$, for $1 \leqslant j \leqslant n-1$, is defined as:
\[
 {R}_j g(x) = \frac{2}{\omega_n} \int_{\mathbb{R}^{n-1}} \frac{x_j - u_j}{|x - u|^n} g(u) \, \mathrm{d}u,
\]
and the Poisson kernel, $P_t(x)$, along with its conjugate $Q_t^{(j)}(x)$ for $j=1, \ldots, n-1$, are given by:
\begin{align*}
	P_t(x) &= \frac{2}{\omega_n} \frac{t}{[t^2 + |x|^2]^{n/2}}, \\
	Q_t^{(j)}(x) &= \frac{2}{\omega_n} \frac{x_j}{[t^2 + |x|^2]^{n/2}}.
\end{align*}

Through convolution, we have:
\begin{align*}
	\mathcal{C}f(z) &= \frac{1}{2} P_t * f(x) + \frac{1}{2} \sum_{j=1}^{n-1} \mathbf{e}_0\mathbf{e}_j Q_t^{(j)} * f(x) \\
	&= \frac{1}{2} P_t * \Big( \big(\mathrm{Id} + \mathbf{e}_0\sum_{j=1}^{n-1} \mathbf{e}_j   R_j \big) f\Big) (x),
\end{align*}

In higher dimensions, the Clifford-Riesz-Hilbert  transform  serves   as a natural generalization of the Hilbert transform, defined as:
\[
\mathcal{H} = \sum_{j=1}^{n-1} \mathbf{e}_j  R_j.
\]
The Cauchy integral operator can be expressed as:
\[
\mathcal{C} = \frac{1}{2} P_t * (\mathrm{Id} + \mathbf{e}_0\mathcal{H}),
\]
highlighting the operator's dependence on the properties of the Poisson kernel and the $L^p$-boundedness of the Riesz transforms, as elucidated in \cite{GTM249,SW71}.

This framework underscores the Cauchy integral operator's crucial role in facilitating the analytic continuation of Clifford-analytic functions from their boundary values within higher-dimensional

This construction  underscores the Cauchy integral operator's essential function in facilitating the analytic continuation of Clifford-analytic functions from their boundary values within higher-dimensional Hardy spaces.

The properties of the Poisson kernel, alongside the $L^p$-bounded nature of the Riesz transforms \cite{GTM249,SW71}, culminate in the following significant theorem:

\begin{theorem}\cite[P.122 Theorem 5.16]{Gilbert}\label{thm:Hp}
	Let $f$ be a function in $L^p(\mathbb{R}^{n-1}, \mathfrak{H})$ for $1 < p < \infty$, or let $f$ and $\mathcal{H}f$ be in $L^1(\mathbb{R}^{n-1}, \mathfrak{H})$ when $p = 1$. The function $\mathcal C f$ then resides in $H^p(\mathbb{R}_{+}^n, \mathfrak{H})$, and for almost every $x \in \mathbb{R}^{n-1}$, it follows that
	\[
\mathcal 	B(\mathcal  C f(z))  = \frac{1}{2}( \mbox{Id} + \mathbf{e}_0\mathcal{H}) f(x).
	\]

	Conversely, if $F$ is a function in $H^p(\mathbb{R}_{+}^n, \mathfrak{H})$ for any $1 \leqslant  p < \infty$, $F$ can be represented as $F =\mathcal C \mathcal B(F)$, where $\mathcal B$ is defined  by \eqref{eq:opB}.
\end{theorem}

This theorem allows for the definition of the Cauchy integral operator for $p > 1$,
\[\mathcal C: L^p(\mathbb{R}^{n-1}, \mathfrak{H}) \to H^p(\mathbb{R}_{+}^n, \mathfrak{H}),\]
as specified  in \eqref{eq:Cauchy}.

Analogously to \eqref{eq:RealHp}, we introduce
\[\mathbf{R}H^p(\mathbb{R}_{+}^n, \mathfrak{H}) = \{f \in L^p(\mathbb{R}^{n-1}, \mathfrak{H}) : f = \mathbf{e}_0\mathcal{H}f\},\]
thus establishing an isomorphism
\[H^p(\mathbb{R}_{+}^n, \mathfrak{H}) \cong \mathbf{R}H^p(\mathbb{R}_{+}^n, \mathfrak{H}),\]
indicating the equivalence of Hardy spaces in Clifford analysis with their real $L^p$ analogs under the Clifford module framework.

\subsection{Gilbert's Conjecture}

A pivotal question in the field of Clifford analysis concerns the existence of a Real \(H^p\) theory for Hardy spaces valued in Clifford modules. This inquiry is encapsulated in what is known as Gilbert's conjecture, documented in Gilbert's seminal work on Clifford analysis \cite[P.140 Conjecture 7.23]{Gilbert}.

\begin{conjecture}\cite{Gilbert}\label{conj:GC}
	Given any Clifford module \(\mathfrak{H}\), there exists an element \(\eta \in \mathbb{R}^n\) satisfying \(\eta^2 = -1\), and a subspace \(\mathfrak{H}_0 \subseteq \mathfrak{H}\), such that:
	\begin{enumerate}
		\item The module \(\mathfrak{H}\) decomposes into \(\mathfrak{H} = \mathfrak{H}_0 \oplus \eta\mathfrak{H}_0\),
		\item The Cauchy integral operator \(C\) forms an isomorphism from \(L^p(\mathbb{R}^{n-1}, \mathfrak{H}_0)\) to \(H^p(\mathbb{R}_{+}^n, \mathfrak{H})\) for all \(p > 1\),
		\item The boundary operator  \( \mathcal B\) and followed by an operator \(\mathcal{R}\)  which is two times the orthogonal projection  onto \(\mathfrak{H}_0\), maps functions from \(H^p(\mathbb{R}_{+}^n, \mathfrak{H})\) continuously onto \(L^p(\mathbb{R}^{n-1}, \mathfrak{H}_0)\) for all \(p > 1\).
	\end{enumerate}
\end{conjecture}

This conjecture suggests that for each Clifford module-valued Hardy space, there exists a corresponding subspace \(\mathfrak{H}_0\) within \(\mathfrak{H}\), such that the sequence of mappings
\begin{equation}
	L^p(\mathbb{R}^{n-1}, \mathfrak{H}_0 ) \hookrightarrow L^p(\mathbb{R}^{n-1}, \mathfrak{H}) \xrightarrow{\mathcal{C}} H^p(\mathbb{R}^n_+, \mathfrak{H}) \xrightarrow{\mathcal{B}} L^p(\mathbb{R}^{n-1}, \mathfrak{H}) \xrightarrow{\mathcal{R}} L^p(\mathbb{R}^{n-1}, \mathfrak{H}_0)
\end{equation}
results in the identity on \(L^p(\mathbb{R}^{n-1}, \mathfrak{H}_0)\). Similarly, the sequence of operator compositions
\begin{equation}
	H^p(\mathbb{R}_{+}^n, \mathfrak{H}) \xrightarrow{\mathcal{B}} L^p(\mathbb{R}^{n-1}, \mathfrak{H}) \xrightarrow{\mathcal{R}} L^p(\mathbb{R}^{n-1}, \mathfrak{H}_0) \hookrightarrow L^p(\mathbb{R}^{n-1}, \mathfrak{H}) \xrightarrow{\mathcal{C}} H^p(\mathbb{R}_{+}^n, \mathfrak{H})
\end{equation}
acts as the identity on \(H^p(\mathbb{R}_{+}^n, \mathfrak{H})\).

This leads to a necessary condition for the Gilbert conjecture to hold.

\medskip

\noindent{\bf Operator Condition:} If the Gilbert conjecture is true, then
\begin{equation}
	\label{Eq:opid}
	\mathcal{R}\mathcal{B}\mathcal{C} = {\mathop{\mathrm{Id}}}_{L^p(\mathbb{R}^{n-1}, \mathfrak{H}_0 )}, \quad \mathcal{C}\mathcal{R}\mathcal{B} = {\mathop{\mathrm{Id}}}_{H^p(\mathbb{R}_{+}^n, \mathfrak{H})}.
\end{equation}

Existing evidence supports Gilbert's conjecture in specific instances:
\begin{itemize}
	\item Gilbert himself demonstrated the conjecture's validity for \(\mathfrak{H} =  C\ell_n\), setting \(\mathfrak{H}_0 = C\ell_{n-1}\) and \(\eta = \mathbf{e}_0\) \cite{GM88}, \cite[P.125]{Gilbert}.
	\item For \(n = 8\), the conjecture is affirmed through an exact realization of \(C\ell_8\) and its spinor space, leveraging the properties of the octonion algebra. However, the conjecture is invalidated for octonionic Hardy spaces \cite{Li23}.
\end{itemize}

 Importantly, the Clifford algebra $C\ell_{8k+m}$ is identified as semi-simple, being expressible as the sum involving the automorphism of a finite vector space. In cases where the algebra is simple, this vector space corresponds precisely to the spinor space.

The spinor spaces $\mathcal{R}_n$ serve as the irreducible modules of the Clifford algebra $C\ell_n$. It can be demonstrated that Gilbert's conjecture is valid for any Clifford module if and only if it is valid for any irreducible Clifford module $\mathcal{R}_n$.

With the elucidation of the above operator condition, resolving Gilbert's conjecture depends on formulating an equivalent condition:

\medskip

\noindent{\bf Algebraic Condition:} Gilbert's conjecture concerning the spinor space $\mathcal{R}_n$ is fulfilled if and only if one can identify an element $\eta \in \mathbb{R}^n$ with $\eta^2 = -1$, and a subspace $\mathfrak{H}_0 \subset \mathcal{R}_n$, fulfilling the criteria that:
\begin{itemize}
	\item The Spinor space $\mathcal{R}_n$ decomposes into $\mathfrak{H}_0 \oplus \eta\mathfrak{H}_0$,
	\item For each $j = 1, \ldots, n-1$, the relations $\eta\mathbf{e}_0\mathbf{e}_j\mathfrak{H}_0 = \mathbf{e}_0\mathbf{e}_j\eta\mathfrak{H}_0 = \mathfrak{H}_0$ are satisfied.
\end{itemize}
Furthermore, if $\eta$ can be chosen as $\mathbf{e}_0$, the condition simplifies to $\mathbf{e}_j\mathfrak{H}_0 = \mathfrak{H}_0$ for each $j = 1, \ldots, n-1$.

This criterion establishes $\mathfrak{H}_0$ as a module for $C\ell_{n-2}$, leading to a significant revelation:

The existence of such an $\eta$ and subspace $\mathfrak{H}_0$ underpins the dimensional constraint:
\[
\dim_{\mathbb{R}}\mathcal{R}_n \geqslant 2 \dim_{\mathbb{R}}\mathcal{R}_{n-2}.
\]
Specifically, Gilbert's conjecture is inapplicable when $n = 8k+6$ or $8k+7$, where $k \in \mathbb{N}$.

Conversely, for the cases not enumerated above, an explicit construction of $C\ell_n$ and the spinor space $\mathcal{R}_n$ is facilitated using the division algebras $\mathbb{R}$, $\mathbb{C}$, $\mathbb{H}$, and $\mathbb{O}$, thereby substantiating Gilbert's conjecture for $n \neq 8k+6$ and $8k+7$, for any natural number $k$.

The proof is solidly based on the intricate constructions of the Clifford algebra \(C\ell_n\) and its associated spinor spaces \(\mathcal{R}_n\). This involves a meticulous selection of \(\eta\) from \(\mathbb{R}^n\) and a particular subspace \(\mathfrak{H}_0\) within \(\mathcal{R}_n\).

 For the sake of clarity and thoroughness, we have meticulously arranged the descriptions and details into a tabular format; see Table 1. Moving forward, we will proceed under the assumption that
  $k\geqslant 1$.

\begin{table}
 	\begin{center}
		\begin{tabular}[ht]{c c c c c c}
			\hline
			$m$ & $C\ell_{8k+m}$ & $\mathcal{R}_{8k+m}$&$\eta$ &$\dim_{\R}\mathcal{R}_{m}$ &$\dim_{\R}\mathcal{R}_{8k+m}$\\
			\hline
			0 & $End_{\mathbb{R}}\big(\left(\mathbb{O}^2\right)^{\otimes k}\big)$ &  $\left(\mathbb{O}^2\right)^{\otimes k}$&$\bfe_0$ &1 &$2^{4k}$\\
			1 & $End_{\mathbb{C}}\big(\mathbb{C}\otimes\left(\mathbb{O}^2\right)^{\otimes k}\big)$ &  $\mathbb{C}\otimes\left(\mathbb{O}^2\right)^{\otimes k}$&$\bfe_0$& 2&$2^{4k+1}$\\
			2 & $End_{\mathbb{H}}\big(\mathbb{H}\otimes\left(\mathbb{O}^2\right)^{\otimes k}\big)$ &  $\mathbb{H}\otimes\left(\mathbb{O}^2\right)^{\otimes k}$&$\bfe_0$& 4&$2^{4k+2}$\\
			3 &  $End_{\mathbb{H}}(\mathbb{H}^+\otimes\left(\mathbb{O}^2\right)^{\otimes k})\oplus End_{\mathbb{H}}(\mathbb{H}^-\otimes\left(\mathbb{O}^2\right)^{\otimes k})$
			&$\mathbb{H}^\pm\otimes\left(\mathbb{O}^2\right)^{\otimes k}$&$\bfe_1$&4&$2^{4k+2}$\\
			4 & $End_{\mathbb{H}}\left(\mathbb{H}^2\otimes\left(\mathbb{O}^2\right))^{\otimes k}\right)$ & $\mathbb{H}^2\otimes\left(\mathbb{O}^2\right)^{\otimes k}$&$\bfe_0$& 8&$2^{4k+3}$\\
			5 & $End_{\mathbb{C}}\left(\mathbb{H}^2\otimes\left(\mathbb{O}^2\right))^{\otimes k}\right)$ & $\mathbb{H}^2\otimes\left(\mathbb{O}^2\right)^{\otimes k}$&$\bfe_1$& 8&$2^{4k+3}$\\
			6 & $End_{\mathbb{R}}\left(\mathbb{H}^2\otimes\left(\mathbb{O}^2\right))^{\otimes k}\right)$ & $\mathbb{H}^2\otimes\left(\mathbb{O}^2\right)^{\otimes k}$&-&8&$2^{4k+3}$\\
			7 & $End_{\mathbb{R}}(\mathbb{O}^+\otimes\left(\mathbb{O}^2\right)^{\otimes k})\oplus End_{\mathbb{R}}(\mathbb{O}^-\otimes\left(\mathbb{O}^2\right)^{\otimes k})$
			&$\mathbb{O}^\pm\otimes\left(\mathbb{O}^2\right)^{\otimes k}$&-&8&$2^{4k+3}$ \\
			\hline
		\end{tabular}
	\end{center}
\bigskip
\caption{The explicit construction of   $C\ell_n$, $\mathcal R_n$, and the choice of $\eta$ }\label{Tab:1}
\end{table}

 \bigskip

The organization of this paper is as follows. Section 2 revisits foundational concepts in Clifford algebras and octonion algebra. Section 3 presents certain identities related to octonionic right multiplication operators, setting the stage for Section 4. In Section 4, we detail an explicit construction of the Clifford algebra \(C\ell_n\) and the associated spinor space \(\mathcal{R}_n\). Section 5 is dedicated to a re-examination of Gilbert’s Conjecture, laying the groundwork for subsequent analyses. Building on this reformulation, Section 6 demonstrates that Gilbert’s Conjecture is not valid for \(n=8k+6, 8k+7\), whereas in Section 7, we establish its validity for cases when \(n \neq 8k + 6\) or \(8k + 7\). The concluding section offers insights into the selection of \(\eta\) and elaborates on the structure of \(\mathfrak{H}_0\).

\section{Preliminaries}

In this section, we establish the notation and conventions while summarizing the foundational aspects of Clifford algebras and octonions.

\subsection{Universal Clifford Algebra \( C\ell_n\) over \(\mathbb{R}^n\)}

For a thorough understanding of Clifford algebra, we encourage the readers to consult the authoritative texts \cite{Atiyah, Gilbert, F.S, BDS82, Harvey90}.

\begin{definition}
	Let \(\mathbb{A}\) be an associative algebra over \(\mathbb{R}\) endowed with a multiplicative identity, and consider \(v: \mathbb{R}^n \rightarrow \mathbb{A}\) to be a linear mapping from \(\mathbb{R}^n\) into \(\mathbb{A}\). The pair \((\mathbb{A}, v)\) constitutes the universal Clifford algebra \( C\ell_n\) associated with \(\mathbb{R}^n\) if the following criteria are met:
	\begin{enumerate}[label=(\arabic*)]
		\item Algebra \(\mathbb{A}\) is algebraically generated by \(\{v(x) : x \in \mathbb{R}^n\}\) and \(\{\lambda1 : \lambda \in \mathbb{R}\}\),
		\item The square of the image of any vector \(x \in \mathbb{R}^n\) under \(v\), i.e., \((v(x))^2\), equates to the negative of the squared norm of \(x\), symbolically \((v(x))^2 = -\lVert {x} \rVert^2\),
		\item The real dimension of \(\mathbb{A}\), expressed as \(\dim_{\mathbb{R}}\mathbb{A}\), equals \(2^n\).
	\end{enumerate}
\end{definition}

\textbf{Some notations and conventions}
\begin{itemize}
 \item\label{g_i} Define the vectors
 \[\mathbf{e}_{i-1} = (0, \ldots, 0, 1, 0, \ldots, 0)\]
 for \(i = 1, \ldots, n\), where the \(1\) is located in the $i$-th position for each \(i\). This collection, \(\mathbf{e}_0, \ldots, \mathbf{e}_{n-1}\), constitutes the standard orthonormal basis of \(\mathbb{R}^n\).  We make no distinction between \(\mathbf{e}_i\) in \(\mathbb{R}^n\) and its corresponding image \(v(\mathbf{e}_i)\) in \(C\ell_n\), treating them as identical.

  \item We denote
  $$\mathbf{x}=\sum_{i=0}^{n-1}x_i\bfe_i \in \R^n,  \quad  \lVert \mathbf{x} \rVert^2 =\sum_{i=0}^{n-1}x_i^2;$$

  \item Let \(\mathcal{P}(n)\) denote the power set of \(\{0, \ldots, n-1\}\) that encompasses the empty set alongside all non-empty subsets \(\alpha = \{\alpha_1, \ldots, \alpha_k\}\) satisfying \(0 \leqslant \alpha_1 < \cdots < \alpha_k \leqslant n-1\). For a given subset \(\alpha\), we define \(\mathbf{e}_{\alpha} = \mathbf{e}_{\alpha_1} \cdots \mathbf{e}_{\alpha_k}\), and assign \(\mathbf{e}_{\emptyset} = 1\) for the empty set.

  \item  The universal Clifford algebra \(C\ell_n\) is generated as a real linear space by the set \(\{\mathbf{e}_\alpha :  \alpha \in \mathcal{P}(n)\}\).

    \item   The multiplication rule in \(C\ell_n\) is defined by the relation $$\mathbf{e}_i\mathbf{e}_j + \mathbf{e}_j\mathbf{e}_i = -2\delta_{ij}$$  for all \(i, j = 0, \ldots, n-1\), where \(\delta_{ij}\) is the Kronecker delta.

    \item Clifford conjugation is a linear operation within the Clifford algebra, defined on the basis elements by the formula
    \[
    \bar{\mathbf{e}}_\alpha = (-1)^{\frac{|\alpha|(|\alpha|+1)}{2}} \mathbf{e}_\alpha,
    \]
    where the magnitude \(|\alpha|\) for a subset \(\alpha = \{\alpha_1, \ldots, \alpha_k\}\) is determined by \(|\alpha| = k\).
       \end{itemize}
The Clifford algebra \(C\ell_n\) is renowned for its semi-simple nature, as detailed in the following theorem.

\begin{theorem}\label{thm:semi-simple}
	For different values of \(n\), the algebra \(C\ell_n\) is isomorphic to the endomorphism algebra of certain vector spaces, as described below:
	\begin{enumerate}
		\item For \(n = 8k\) or \(n = 8k+6\), there exists a real vector space \(\mathcal{R}_{n}\) such that $$ C\ell_{n} \cong \text{End}_{\mathbb{R}}(\mathcal{R}_{n}).$$
		\item For \(n = 8k+1\) or \(n = 8k+5\), there exists a complex vector space \(\mathcal{R}_{n}\) such that $$ C\ell_{n} \cong \text{End}_{\mathbb{C}}(\mathcal{R}_{n}).$$
		\item For \(n = 8k+2\) or \(n = 8k+4\), there exists a left quaternionic vector space \(\mathcal{R}_{n}\) such that $$ C\ell_{n} \cong \text{End}_{\mathbb{H}}(\mathcal{R}_{n}).$$
		\item For \(n = 8k+3\) or \(n = 8k+7\), it is established that $$ C\ell_{8k+3} \cong \text{End}_{\mathbb{H}}(\mathcal{R}^+_{8k+3}) \oplus \text{End}_{\mathbb{H}}(\mathcal{R}^-_{8k+3})$$  and $$C\ell_{8k+7} \cong \text{End}_{\mathbb{R}}(\mathcal{R}^+_{8k+7}) \oplus \text{End}_{\mathbb{R}}(\mathcal{R}^-_{8k+7}).$$
	\end{enumerate}
\end{theorem}

\begin{definition}\label{def:spinor}
  The vector space $\mathcal{R}_{n}$ in the theorem above is called the spinor spaces.
\end{definition}

\begin{remark}
	The designation ``spinor space" varies across literature. Harvey refers to the vector spaces \(\mathcal{R}_{n}\) as ``pinor spaces" \cite[P.210]{Harvey90}, whereas Gilbert labels them as "real spinor spaces" \cite[P.59 (7.42)]{Gilbert}. Regardless of the nomenclature, these spaces represent the class of irreducible \(C\ell_n\)-modules.
\end{remark}

\begin{definition}\label{def:CM}
	A finite-dimensional real Hilbert space \(\mathfrak{H}\) is termed a \(C\ell_n\) module if there exist skew-adjoint real linear operators \(T_1, \ldots, T_n\) satisfying the relations
	\[T_j T_k + T_k T_j = -2\delta_{jk} \text{Id}\] for any $1 \leqslant j, k \leqslant  n.$
\end{definition}

  \subsection{Octonion algebra}

The octonions \(\mathbb{O}\) form a non-commutative, non-associative normed division algebra of eight dimensions over \(\mathbb{R}\). For a comprehensive discussion on octonion algebra, refer to \cite{Baez,CS03}.

\textbf{Some notations and conventions}
\begin{itemize}
\item Let \(\mathbf{e_0} = 1, \mathbf{e_1}, \ldots, \mathbf{e_7}\) denote the canonical basis of the octonions. The multiplication of basis elements adheres to the rule
\[
\mathbf{e_i}\mathbf{e_j} + \mathbf{e_j}\mathbf{e_i} = -2\delta_{ij}, \quad \text{for } i, j = 1, \ldots, 7.
\]
Note that in \(\mathbb{R}^8\), the notation \(\mathbf{e_i}\) might represent a vector in \(\mathbb{R}^8\), a Clifford algebra element in \(Cl_8\), or an octonion in \(\mathbb{O}\). These representations are not differentiated unless clarification is necessary to prevent confusion.

\item For any octonion \(x = x_0 + \sum_{i=1}^7 x_i \mathbf{e_i} \in \mathbb{O}\), with \(x_i \in \mathbb{R}\), the octonion conjugation is defined as \(\overline{x} = x_0 - \sum_{i=1}^7 x_i \mathbf{e_i}\), and the real part operator is given by \(Re(x) = x_0\). The set \(\mathbb{O}\) constitutes an 8-dimensional Euclidean space equipped with the inner product \((p, q) = Re(p\overline{q})\) for any \(p, q \in \mathbb{O}\).

\item The associator is  defined for any elements \(a, b, c\in\O \) by
\[
[a, b, c] = (ab)c - a(bc).
\]
In the context of octonions, the associator exhibits alternativity, satisfying the identities
\[
[a, b, c] = -[a, c, b] = -[b, a, c] = -[\overline{a}, b, c].
\]

\item The intricacies of the octonions' multiplication table are neatly captured by the Fano plane, as depicted in Figure 1 and elaborated upon in \cite{Baez}. Within this schema, the Fano plane's vertices are marked by the imaginary units \(\mathbf{e_1}, \ldots, \mathbf{e_7}\). Each of the plane's  directed lines or the circle delineates a quaternionic triple. The multiplication of any two distinct imaginary units is defined by the third unit situated on the  line or the circle linking them, with the multiplication's sign being influenced by the line's directional orientation.

\begin{figure}[ht]
\centering
  \includegraphics[width=6cm]{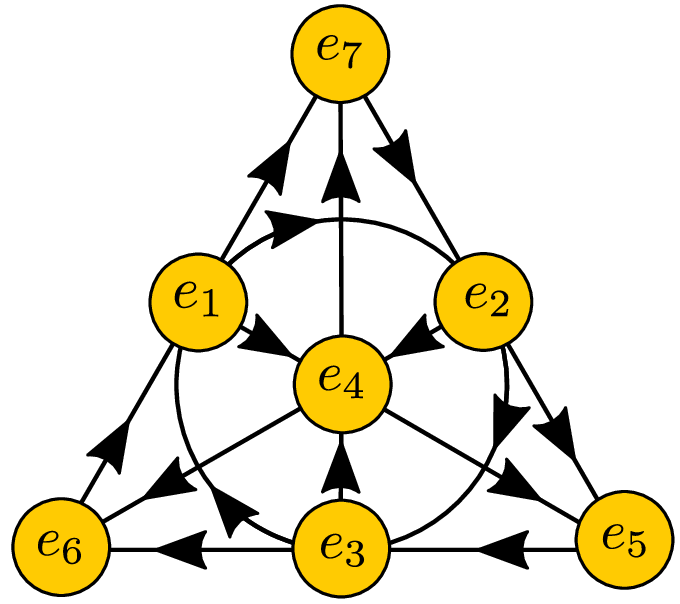}
  \caption{Fano plane}
  \label{fig:1}
\end{figure}

\end{itemize}

\section{Octonionic Techniques}

 Our investigation into Clifford algebras and spinor spaces prominently features the octonions as a fundamental instrument. At the heart of our approach is the formulation of matrices. The elements of these matrices are operators, crafted from the act of right multiplication by octonions.

For any octonion number \(p \in \mathbb{O}\), we introduce an operator \(M_p\) that is classified within the set of endomorphisms of \(\mathbb{O}^2\) over the real numbers, explicitly denoted by \(M_p \in \text{End}_{\mathbb{R}}(\mathbb{O}^2)\). The operator is defined through the matrix representation
\[
M_p = \begin{pmatrix}
	0 & R_p \\
	-R_{\bar{p}} & 0
\end{pmatrix},
\]
where each matrix component is a mapping applicable to column vectors from \(\mathbb{O}^2\). In this context, \(R_p \in \text{End}_{\mathbb{R}}(\mathbb{O})\) signifies the right multiplication operator, described by the mapping
\[
R_p: \mathbb{O} \to \mathbb{O}, \quad R_p(q) = qp,
\]
for any element \(q \in \mathbb{O}\). This setup provides a structured method for engaging with elements of \(\mathbb{O}^2\) via the octonion algebra, specifically through the utilization of \(M_p\).

Additionally, we define the matrix
\[
E = \begin{pmatrix}
	Id_{\mathbb{O}} & 0 \\
	0 & -Id_{\mathbb{O}}
\end{pmatrix}.
\]

The following lemma plays a crucial role in the construction of \(C\ell_n\) and \(\mathcal{R}_n\):

\begin{lemma}
	For any \(p, q \in \mathbb{O}\), the subsequent properties hold:
	\begin{enumerate}
		\item \(R_pR_{\bar{q}} + R_qR_{\bar{p}} = 2(p,q)Id_{\mathbb{O}}.\)
		\item \(M_pM_{q} + M_qM_{p} = -2(p,q)Id_{\mathbb{O}^2}.\)
		\item \(M_pE + EM_p = 0.\)
	\end{enumerate}
\end{lemma}
\begin{proof}
We just give a proof of (1),  and the simple fact just follows that for any octonion $z$, we have
	\begin{align*}
		&\left(R_pR_{\bar{q}}+R_qR_{\bar{p}}\right)(z)=(z\bar{q})p+(z\bar{p})q\\
		=&[z, \bar{q}, p]+z(\bar{q}p)-[z, \bar{p}, q]+z(\bar{p}q)=z(\bar{q}p+\bar{p}q)\\
=&2zRe(\overline{q}p)=2(p,q)z.
	\end{align*}
\end{proof}

For every octonion $p \in \mathbb{O}$, we associate maps $A_p^s, B_p^t \in \mathrm{End}_{\mathbb{R}}((\mathbb{O}^2)^{\otimes k})$ for $s, t = 1, \ldots, k$, defined as
\[
A_p^s := \underbrace{\mathrm{Id}_{\mathbb{O}^2} \otimes \cdots \otimes \mathrm{Id}_{\mathbb{O}^2} \otimes M_{p}}_{\text{$s$-th position}} \otimes E \otimes \cdots \otimes E.
\]
Furthermore, we define
\[
E_k = \underbrace{E \otimes \cdots \otimes E}_{k \text{ times}} \in \mathrm{End}_{\mathbb{R}}((\mathbb{O}^2)^{\otimes k}).
\]

\begin{lemma}\label{lemma:ABst}
	For all $p, q \in \mathbb{O}$, the following relations hold:
	\begin{enumerate}
		\item For $s \neq t, s, t = 1, \ldots, k$, we have:
		\[
		A_p^s A_q^t + A_q^t A_p^s = 0.
		\]
		\item For $s = 1, \ldots, k$, it follows that:
		\[
		A_p^s A_q^s + A_q^s A_p^s = -2\left( p, q \right) \mathrm{Id}_{(\mathbb{O}^2)^{\otimes k}}.
		\]
		\item For $s = 1, \ldots, k$, we also find:
		\[
		A_p^s E_k + E_k A_p^s = 0.
		\]
	\end{enumerate}
\end{lemma}

\begin{proof}
	For item (1), let's consider the case when \(1 \leqslant s < t \leqslant k\). The composition \(A_p^sA_q^t\) can be detailed as follows:
	\begin{align*}
		A_p^sA_q^t &= \underbrace{\mathrm{Id}_{\mathbb{O}^2} \otimes \cdots \otimes \mathrm{Id}_{\mathbb{O}^2}}_{\text{$s-1$ times}} \otimes M_{p} \otimes \underbrace{E \otimes \cdots \otimes E}_{\text{$t-s-1$ times}} \otimes EM_{q} \otimes \underbrace{E^2 \otimes \cdots \otimes E^2}_{\text{$k-t$ times}} \\
		&= \underbrace{\mathrm{Id}_{\mathbb{O}^2} \otimes \cdots \otimes \mathrm{Id}_{\mathbb{O}^2}}_{\text{$s-1$ times}} \otimes M_{p} \otimes \underbrace{E \otimes \cdots \otimes E}_{\text{$t-s-1$ times}} \otimes EM_{q} \otimes \underbrace{\mathrm{Id}_{\mathbb{O}^2} \otimes \cdots \otimes \mathrm{Id}_{\mathbb{O}^2}}_{\text{$k-t$ times}}.
	\end{align*}
	
	Conversely, calculating \(A_q^tA_p^s\) yields:
	\begin{align*}
		A_q^tA_p^s &= \underbrace{\mathrm{Id}_{\mathbb{O}^2} \otimes \cdots \otimes \mathrm{Id}_{\mathbb{O}^2}}_{\text{$s-1$ times}} \otimes M_{p} \otimes \underbrace{E \otimes \cdots \otimes E}_{\text{$t-s-1$ times}} \otimes M_{q}E \otimes \underbrace{E^2 \otimes \cdots \otimes E^2}_{\text{$k-t$ times}} \\
		&= \underbrace{\mathrm{Id}_{\mathbb{O}^2} \otimes \cdots \otimes \mathrm{Id}_{\mathbb{O}^2}}_{\text{$s-1$ times}} \otimes M_{p} \otimes \underbrace{E \otimes \cdots \otimes E}_{\text{$t-s-1$ times}} \otimes (-EM_{q}) \otimes \underbrace{\mathrm{Id}_{\mathbb{O}^2} \otimes \cdots \otimes \mathrm{Id}_{\mathbb{O}^2}}_{\text{$k-t$ times}} \\
		&= -A_p^sA_q^t.
	\end{align*}
	
	This demonstrates that \(A_p^sA_q^t + A_q^tA_p^s = 0\), confirming the anticommutativity for distinct indices \(s\) and \(t\) as claimed.
\end{proof}

\section{ Realizing $C\ell_n$ and $\mathcal{R}_n$ Through Real Division Algebras }

In this section, we aim to furnish Clifford algebras with explicit matrix realizations. This approach will enable us to derive concrete expressions for the spinor spaces associated with these algebras.

The construction of explicit matrix realizations for Clifford algebras is based on the universal property of Clifford algebras, which we will detail below.

\begin{theorem}[{\cite[P.31, Theorem 3.17]{Gilbert}}]\label{thm:construction}
	Let $\mathbb{A}$ be an associative algebra over $\mathbb{R}$ with identity element $1$. Consider a map $v: \mathbb{R}^n \rightarrow \mathbb{A}$ that is an $\mathbb{R}$-linear embedding. Assume the following conditions are met:
	\begin{enumerate}
		\item The algebra $\mathbb{A}$ is generated by the set $\{v(x) : x \in \mathbb{R}^n\}$ along with $\{\lambda1 : \lambda \in \mathbb{R}\}$.
		\item For every $x \in \mathbb{R}^n$, it holds that $(v(x))^2 = -\lVert x \rVert^2$.
	\end{enumerate}
	Then, exactly one of the following statements is true:
	\begin{enumerate}
		\item The dimension of $\mathbb{A}$ over $\mathbb{R}$, denoted $\dim_{\mathbb{R}}\mathbb{A}$, is $2^n$.
		\item If $n \equiv 3 \ (\text{mod} \ 4)$, then $\dim_{\mathbb{R}}\mathbb{A} = 2^{n-1}$, and the product $\mathbf{e}_0\mathbf{e}_1\cdots\mathbf{e}_{n-1}$ belongs to $\mathbb{R}$.
	\end{enumerate}
\end{theorem}

We present an explicit matrix realization of $C\ell_n$ and consequently derive the realization of the spinor space $\mathcal{R}_n$, utilizing real division algebras.

\begin{theorem}\label{THm:main-988}
	Real division algebras are employed to furnish an explicit matrix realization of $C\ell_n$ and to articulate the corresponding realization of the spinor space $\mathcal{R}_n$. These realizations are tabulated in the introduction.
		 \end{theorem}

\begin{remark} \label{thm:CLs}
It should be noted that the realization can be detailed as follows.

	\begin{enumerate}
		\item For the real vector space $V_{8k} = \mathbb{R}^{8k} = \mathbb{O}^k$, we define a $\mathbb{R}$-linear embedding
		\[
		v_{8k}: V_{8k} \to \mathrm{End}_{\mathbb{R}}\left((\mathbb{O}^2)^{\otimes k}\right)
		\]
		specified by
		\[
		v_{8k}(p_1, p_2, \ldots, p_k) = \sum_{j=1}^k A_{p_j}^j,
		\]
		for $p_i \in \mathbb{O}, i = 1, \ldots, k$.

		This embedding $v_{8k}$ provides a realization of the Clifford algebra $C\ell_{8k}$, thereby establishing $(\mathbb{O}^2)^{\otimes k}$ as the real spinor space $\mathcal{R}_{8k}$.
		
		\item For the real vector space $V_{8k+1} = \mathbb{R}^{8k+1} = \mathbb{R}\oplus\mathbb{O}^{k}  $, we introduce a $\mathbb{R}$-linear embedding
		\[
		v_{8k+1}: V_{8k+1} \to \mathrm{End}_{\mathbb{C}}(\mathbb{C} \otimes (\mathbb{O}^2)^{\otimes k})
		\]
		defined by
		\[
		v_{8k+1}(r, p_1, p_2, \ldots, p_k) = L_{\mathrm{i}r} \otimes E_k + \sum_{j=1}^k \mathrm{Id}_{\mathbb{C}} \otimes A_{p_j}^j,
		\]
		where $p_j \in \mathbb{O}$ for $j = 1, \ldots, k$, and $r \in \mathbb{R}$. Here, $L_{\mathrm{i}r} \in \mathrm{End}_{\mathbb{C}}(\mathbb{C})$ denotes the left multiplication operator by $\mathrm{i}r$, which acts as $L_{\mathrm{i}r}(z) = \mathrm{i}rz$ for all $z \in \mathbb{C}$.

		This construction of $v_{8k+1}$ furnishes a realization of the Clifford algebra $C\ell_{8k+1}$, thereby identifying $\mathbb{C} \otimes (\mathbb{O}^2)^{\otimes k}$ as the corresponding spinor space $\mathcal{R}_{8k+1}$.

		\item For the real vector space $V_{8k+2} = \mathbb{R}^{8k+2} = \mathbb{R}^2 \oplus \mathbb{O}^{k}$, we establish a $\mathbb{R}$-linear embedding
		\[
		v_{8k+2}: V_{8k+2} \to \mathrm{End}_{\mathbb{H}}(\mathbb{H} \otimes (\mathbb{O}^2)^{\otimes k})
		\]
		defined by
		\[
		v_{8k+2}((x,y), p_1, p_2, \ldots, p_k) = L_{x\mathrm{i} + y\mathrm{j}} \otimes E_k + \sum_{j=1}^k \mathrm{Id}_{\mathbb{H}} \otimes A_{p_j}^j,
		\]
		where $p_i \in \mathbb{O}$ for $i = 1, \ldots, k$, and $x, y \in \mathbb{R}$. Here, $L_{x\mathrm{i} + y\mathrm{j}} \in \mathrm{End}_{\mathbb{H}}(\mathbb{H})$ represents the left multiplication operator by $x\mathrm{i} + y\mathrm{j}$, which acts as $L_{x\mathrm{i} + y\mathrm{j}}(q) = (x\mathrm{i} + y\mathrm{j})q$ for all $q \in \mathbb{H}$.

		This embedding $v_{8k+2}$ yields a realization of the Clifford algebra $C\ell_{8k+2}$, thus allowing $\mathbb{H} \otimes (\mathbb{O}^2)^{\otimes k}$ to be recognized as the spinor space $\mathcal{R}_{8k+2}$.
		
		\item For the real vector space $V_{8k+3} = \mathbb{R}^{8k+3} = \mathrm{Im}\mathbb{H} \oplus \mathbb{O}^{k}$, we introduce a $\mathbb{R}$-linear embedding
		\[
	\qquad \qquad 	v_{8k+3}: V_{8k+3} \to \mathrm{End}_{\mathbb{H}}(\mathbb{H}^+ \otimes (\mathbb{O}^2)^{\otimes k}) \oplus \mathrm{End}_{\mathbb{H}}(\mathbb{H}^- \otimes (\mathbb{O}^2)^{\otimes k})
		\]
		defined by
		\[
		\qquad \qquad  	v_{8k+3}(q, p_1, p_2, \ldots, p_k) = \left(R_q \otimes E_k + \sum_{j=1}^k \mathrm{Id}_{\mathbb{H}} \otimes A_{p_j}^j, R_{\bar{q}} \otimes E_k + \sum_{j=1}^k \mathrm{Id}_{\mathbb{H}} \otimes A_{p_j}^j\right),
		\]
		where $p_i \in \mathbb{O}$ for $i = 1, \ldots, k$, and $q \in \mathrm{Im}\mathbb{H}$.
		
		 This embedding $v_{8k+3}$ establishes a realization of the Clifford algebra $C\ell_{8k+3}$, allowing for $\mathbb{H}^\pm \otimes (\mathbb{O}^2)^{\otimes k}$ to serve as the spinor spaces $\mathcal{R}^\pm_{8k+3}$.

		\item For the real vector space $V_{8k+4} = \mathbb{R}^{8k+4} = \mathbb{H} \oplus \mathbb{O}^{k}$, a $\mathbb{R}$-linear embedding is defined as follows:
		\[
		v_{8k+4}: V_{8k+4} \to \mathrm{End}_{\mathbb{H}}(\mathbb{H}^2 \otimes (\mathbb{O}^2)^{\otimes k})
		\]
		by the mapping
		\[
		v_{8k+4}(q, p_1, p_2, \ldots, p_k) = \left(
		\begin{array}{cc}
			0 & R_q \\
			-R_{\overline{q}} & 0 \\
		\end{array}
		\right) \otimes E_k + \sum_{j=1}^k \mathrm{Id}_{\mathbb{H}^2} \otimes A_{p_j}^j,
		\]
		where $p_i \in \mathbb{O}$ for $i = 1, \ldots, k$, and $q \in \mathbb{H}$.

		This embedding $v_{8k+4}$ facilitates a realization of the Clifford algebra $C\ell_{8k+4}$, and as a result, $\mathbb{H}^2 \otimes (\mathbb{O}^2)^{\otimes k}$ is recognized as the spinor space $\mathcal{R}_{8k+4}$.
		
		\item  For the real vector space $V_{8k+5} = \mathbb{R}^{8k+5} = \mathbb{R} \oplus \mathbb{H} \oplus \mathbb{O}^{k}$, we establish a $\mathbb{R}$-linear embedding:
		\[
		v_{8k+5}: V_{8k+5} \to \mathrm{End}_{\mathbb{C}}(\mathbb{H}^2 \otimes (\mathbb{O}^2)^{\otimes k})
		\]
		utilizing the formulation:
		\[
		v_{8k+5}(r, q, p_1, p_2, \ldots, p_k) = \left(
		\begin{array}{cc}
			L_{\mathrm{i}r} & R_q \\
			-R_{\overline{q}} & -L_{\mathrm{i}r} \\
		\end{array}
		\right) \otimes E_k + \sum_{j=1}^k \mathrm{Id}_{\mathbb{H}^2} \otimes A_{p_j}^j,
		\]
		where $p_i \in \mathbb{O}$ for $i = 1, \ldots, k$, and $q \in \mathbb{H}, r \in \mathbb{R}$.
		
		The embedding $v_{8k+5}$ delivers a realization of the Clifford algebra $C\ell_{8k+5}$, thereby identifying $\mathbb{H}^2 \otimes (\mathbb{O}^2)^{\otimes k}$ as the corresponding spinor space $\mathcal{R}_{8k+5}$.

		\item For the real vector space $V_{8k+6} = \mathbb{R}^{8k+6} = \mathbb{R}^2 \oplus \mathbb{H} \oplus \mathbb{O}^{k}$, a $\mathbb{R}$-linear embedding is defined as follows:
		\[
		v_{8k+6}: V_{8k+6} \to \mathrm{End}_{\mathbb{R}}(\mathbb{H}^2 \otimes (\mathbb{O}^2)^{\otimes k})
		\]
		by the mapping
		\[
	\qquad\qquad 	v_{8k+6}((x,y), q, p_1, p_2, \ldots, p_k) = \left(
		\begin{array}{cc}
			L_{x\mathrm{i} + y\mathrm{j}} & R_q \\
			-R_{\overline{q}} & -L_{x\mathrm{i} + y\mathrm{j}} \\
		\end{array}
		\right) \otimes E_k + \sum_{j=1}^k \mathrm{Id}_{\mathbb{H}^2} \otimes A_{p_j}^j,
		\]
		where each $p_i \in \mathbb{O}$ for $i = 1, \ldots, k$, $q \in \mathbb{H}$, and $x, y \in \mathbb{R}$.
		
		This embedding $v_{8k+6}$ facilitates a realization of the Clifford algebra $C\ell_{8k+6}$, thereby designating $\mathbb{H}^2 \otimes (\mathbb{O}^2)^{\otimes k}$ as the spinor space $\mathcal{R}_{8k+6}$.

		\item For the vector space $V_{8k+7} = \mathbb{R}^{8k+7} = \mathrm{Im}\mathbb{O} \oplus \mathbb{O}^{k}$, a $\mathbb{R}$-linear embedding is specified as follows:
		\[
		v_{8k+7}: V_{8k+7} \to \mathrm{End}_{\mathbb{R}}(\mathbb{O}^+ \otimes (\mathbb{O}^2)^{\otimes k}) \oplus \mathrm{End}_{\mathbb{R}}(\mathbb{O}^- \otimes (\mathbb{O}^2)^{\otimes k})
		\]
		via the function
		\[
	\qquad\qquad 	v_{8k+7}(q, p_1, p_2, \ldots, p_k) = \left(R_q \otimes E_k + \sum_{j=1}^k \mathrm{Id}_{\mathbb{H}} \otimes A_{p_j}^j, R_{\overline{q}} \otimes E_k + \sum_{j=1}^k \mathrm{Id}_{\mathbb{H}} \otimes A_{p_j}^j\right),
		\]
		where $p_i \in \mathbb{O}$ for $i = 1, \ldots, k$, and $q \in \mathrm{Im}\mathbb{O}$.
		
		This embedding $v_{8k+7}$ facilitates the realization of the Clifford algebra $C\ell_{8k+7}$. Consequently, $\mathbb{O}^\pm \otimes (\mathbb{O}^2)^{\otimes k}$ is acknowledged as the spinor space $\mathcal{R}^\pm_{8k+7}$.
			\end{enumerate}

In conclusion, for the ease of our readers, all essential details concerning the space $V_n$
 and the mapping $v_n$
 across all eight cases have been systematically compiled and presented in Table 2.
  \end{remark}

\bigskip

\noindent \textbf{Convention:} \ \ For convenience, we will henceforth adopt the convention that the vector \(\eta\) in \(V_n\) is represented by \(v_n(\eta)\), as detailed in Remark \ref{thm:CLs}. Specifically, this means we identify
\[
\bfe_j \cong v_n(\bfe_j).
\]

\bigskip

\begin{proof}[Proof of Theorem \ref{THm:main-988}]
	Based on Theorem \ref{thm:construction}, for $n \neq 4k + 3$, it suffices to verify that the map $v_n$ fulfills $$v_n(x)^2 = -\|x\|^2$$ for all $x \in \mathbb{R}^n$, and that the algebra generated by the image of \(v_n\) has real dimension \(2^n\), matching the real dimension of the universal Clifford algebra $C\ell_n$.
	
	 This verification is mainly performed for the cases $n = 8k$ and $n = 8k + 3$, suggesting that the methodology for other cases is analogous.
	
	\begin{enumerate}
		\item {\bf Case $\mathbf{n = 8k}$:}

		By applying Lemma \ref{lemma:ABst}, we obtain
		\begin{equation*}
			v_{8k}(p_1, p_2, \ldots, p_k)^2 = \left( \sum_{j=1}^k A_{p_j}^j \right)^2 = -\left( \sum_{j=1}^k \|p_j\|^2 \right) \text{Id}_{(\mathbb{O}^2)^{\otimes k}}.
		\end{equation*}
	Additionally, the dimensionality condition is met:
	\[
	\text{dim}_{\mathbb{R}}\text{End}_{\mathbb{R}}\left((\mathbb{O}^2)^{\otimes k}\right) = 2^{8k} = \text{dim}_{\mathbb{R}}C\ell_{8k}.
	\]
	Thus, the mapping $v_{8k}$ provides a realization of $C\ell_{8k}$ according to Theorem \ref{thm:construction}, and $(\mathbb{O}^2)^{\otimes k}$ is recognized as the spinor space $\mathcal{R}_{8k}$, as outlined in Definition \ref{def:spinor} and Theorem \ref{thm:semi-simple}.
	
 \medskip

    \item {\bf Case  $\mathbf{n = 8k+3}$:}

  Utilizing Lemma \ref{lemma:ABst}, it is deduced that

    \begin{align*}
    \qquad 	v_{8k+3}(q, p_1, p_2, \ldots, p_k)^2 = &\left(R_q \otimes E_k + \sum_{j=1}^{k} Id_{\mathbb{H}} \otimes A_{p_j}^j, R_{\bar{q}} \otimes E_k + \sum_{j=1}^{k} Id_{\mathbb{H}} \otimes A_{p_j}^j\right)^2 \\
    	= &-\left(\|q\|^2 + \sum_{j=1}^k \|p_j\|^2\right) \left(Id_{\mathbb{H} \otimes (\mathbb{O}^2)^{\otimes k}}, Id_{\mathbb{H} \otimes (\mathbb{O}^2)^{\otimes k}}\right).
    \end{align*}
    Moreover, we observe that
    \[
    \bfe_0\bfe_1\cdots\bfe_{n-1} = \left(\varepsilon Id_{\mathbb{H} \otimes (\mathbb{O}^2)^{\otimes k}}, -\varepsilon Id_{\mathbb{H} \otimes (\mathbb{O}^2)^{\otimes k}}\right) \notin \mathbb{R}, \ \ \ \ \varepsilon\in \{1,-1\}.
    \]
    Additionally, it is noted that
    \[
    \text{dim}_{\mathbb{R}}\left(End_{\mathbb{H}}(\mathbb{H}^+ \otimes (\mathbb{O}^2)^{\otimes k}) \oplus End_{\mathbb{H}}(\mathbb{H}^- \otimes (\mathbb{O}^2)^{\otimes k})\right) = 2^{8k+3} = \text{dim}_{\mathbb{R}}C\ell_{8k+3}.
    \]
    Accordingly, the map $v_{8k+3}$ establishes a realization of $C\ell_{8k+3}$ in line with Theorem \ref{thm:construction}, and $\mathbb{H}^\pm \otimes (\mathbb{O}^2)^{\otimes k}$ precisely represents the spinor space $\mathcal{R}_{8k+3}^\pm$, as specified in Definition \ref{def:spinor} and Theorem \ref{thm:semi-simple}.
     \end{enumerate}
\end{proof}

\begin{table}
  \centering
  \begin{tabular}[h]{l|l|l}
  \hline
  $m$ & \hskip20pt $V_{8k+m}$ & \hskip40pt  $v_{8k+m}$ \\
  \hline
  0 & $\mathbb{O}^k$ & $\sum_{j=1}^k A_{p_j}^j$ \\
  1 & $\mathbb{R}\oplus\mathbb{O}^{k}$ & $ L_{\mathrm{i}r} \otimes E_k + \sum_{j=1}^k \mathrm{Id}_{\mathbb{C}} \otimes A_{p_j}^j$ \\
  2 & $\mathbb{R}^2 \oplus \mathbb{O}^{k}$ & $ L_{x\mathrm{i} + y\mathrm{j}} \otimes E_k + \sum_{j=1}^k \mathrm{Id}_{\mathbb{H}} \otimes A_{p_j}^j$ \\
  3 & $\mathrm{Im}\mathbb{H} \oplus \mathbb{O}^{k}$ & $\left(R_q \otimes E_k + \sum_{j=1}^k \mathrm{Id}_{\mathbb{H}} \otimes A_{p_j}^j, R_{\bar{q}} \otimes E_k + \sum_{j=1}^k \mathrm{Id}_{\mathbb{H}} \otimes A_{p_j}^j\right)$ \\
  4 & $\mathbb{H} \oplus \mathbb{O}^{k}$ &$\left(
		\begin{array}{cc}
			0 & R_q \\
			-R_{\overline{q}} & 0 \\
		\end{array}
		\right) \otimes E_k + \sum_{j=1}^k \mathrm{Id}_{\mathbb{H}^2} \otimes A_{p_j}^j$  \\
  5 & $\mathbb{R}\oplus\mathbb{H} \oplus \mathbb{O}^{k}$ &  $\left(
		\begin{array}{cc}
			L_{\mathrm{i}r} & R_q \\
			-R_{\overline{q}} & -L_{\mathrm{i}r} \\
		\end{array}
		\right) \otimes E_k + \sum_{j=1}^k \mathrm{Id}_{\mathbb{H}^2} \otimes A_{p_j}^j$ \\
  6 & $\mathbb{R}^2\oplus\mathbb{H} \oplus \mathbb{O}^{k}$ & $\left(
		\begin{array}{cc}
			L_{x\mathrm{i} + y\mathrm{j}} & R_q \\
			-R_{\overline{q}} & -L_{x\mathrm{i} + y\mathrm{j}} \\
		\end{array}
		\right) \otimes E_k + \sum_{j=1}^k \mathrm{Id}_{\mathbb{H}^2} \otimes A_{p_j}^j$  \\
  7 & $\mathrm{Im}\mathbb{O} \oplus \mathbb{O}^{k}$ &$\left(R_q \otimes E_k + \sum_{j=1}^k \mathrm{Id}_{\mathbb{H}} \otimes A_{p_j}^j, R_{\overline{q}} \otimes E_k + \sum_{j=1}^k \mathrm{Id}_{\mathbb{H}} \otimes A_{p_j}^j\right)$  \\
  \hline
\end{tabular}
\bigskip
  \caption{The vector $V_n$ and the act of $V_{n}$}\label{Tab:2}
\end{table}

\begin{remark}
Here  we detail the explicit spaces of spinor spaces $\mathcal{R}_n$, where the term ``spinor spaces" refers to the irreducible $C\ell_n$-modules. The study of Spinors in $n$-dimensions was pioneered by Brauer and Weyl \cite{BW35}. Furthermore, the construction of what are termed basic spinor and semi-spinor modules, utilizing Lie algebra methods, has been elaborated upon by the authors in   \cite{BJ99}.
\end{remark}


\section{A Reformulation of Gilbert's Conjecture}
In this section, we introduce an alternate formulation of Gilbert's Conjecture that plays a crucial role in addressing the conjecture effectively.

\begin{theorem}\label{thm:redu-134}
	Gilbert's conjecture   holds if and only if  there exists an element \(\eta \in \mathbb{R}^n\) with \(\eta^2 = -1\) and a subspace \(\mathfrak{H}_0\) within \(\mathcal{R}_n\) such that:
	\begin{enumerate}
		\item[\((1)\)] The space \(\mathcal{R}_n\) can be decomposed into \(\mathfrak{H}_0 \oplus \eta\mathfrak{H}_0\),
		\item[\((2)\)] For each \(j = 1, \ldots, n-1\), it holds that \(\eta\mathbf{e}_0\mathbf{e}_j\mathfrak{H}_0 = \mathfrak{H}_0\) and \(\mathbf{e}_0\mathbf{e}_j\eta\mathfrak{H}_0 = \mathfrak{H}_0\).
	\end{enumerate}
	Furthermore, if \(\eta = \mathbf{e}_0\), condition (2) simplifies to \(\mathbf{e}_j\mathfrak{H}_0 = \mathfrak{H}_0\) for \(j = 1, \ldots, n-1\).
\end{theorem}

\begin{proof}[Proof of Theorem \ref{thm:redu-134}]{\textbf{($\Leftarrow$)}}
	First, we assume the existence of \(\eta \in \mathbb{R}^n\) with \(\eta^2 = -1\) and a subspace \(\mathfrak{H}_0 \subseteq \mathcal{R}_n\) such that
	\[\mathcal{R}_n = \mathfrak{H}_0 \oplus \eta\mathfrak{H}_0,\]
	and for every \(j = 1, \cdots, n-1\),
	\[\eta\mathbf{e}_0\mathbf{e}_j\mathfrak{H}_0 = \mathbf{e}_0\mathbf{e}_j\eta\mathfrak{H}_0 = \mathfrak{H}_0.\]
	
	We claim that for any function \(f \in L^p(\mathbb{R}^{n-1}, \mathfrak{H}_0)\),
	\begin{equation*}
		\mathbf{e}_0\mathcal{H}f \in L^p(\mathbb{R}^{n-1}, \eta\mathfrak{H}_0),
	\end{equation*}
	and
	\begin{equation*}
		\mathbf{e}_0\mathcal{H}\eta f \in L^p(\mathbb{R}^{n-1}, \mathfrak{H}_0).
	\end{equation*}
	
	Utilizing the definition of the Clifford-Riesz-Hilbert operator $$\mathcal{H} = \sum_{j=1}^{n-1} \mathbf{e}_j \mathcal R_j,$$ we find
	\begin{equation*}
		\mathbf{e}_0\mathcal{H}f = -\eta^2\mathbf{e}_0\sum_{j=1}^{n-1} \mathbf{e}_j \mathcal R_jf = -\eta\left(\sum_{j=1}^{n-1}\eta\mathbf{e}_0\mathbf{e}_j \mathcal R_jf\right).
	\end{equation*}
	Given that the Riesz transform \(\mathcal R_j\) is \(L^p\) bounded for \(p>1\), it follows that \(\mathcal R_jf \in L^p(\mathbb{R}^{n-1}, \mathfrak{H}_0)\). Combining this with assumption (2), we deduce \(\mathbf{e}_0\mathcal{H}f \in L^p(\mathbb{R}^{n-1}, \eta\mathfrak{H}_0)\) and similarly, \(\mathbf{e}_0\mathcal{H}\eta f \in L^p(\mathbb{R}^{n-1}, \mathfrak{H}_0)\). This establishes the claim.

To demonstrate the assertion in \eqref{Eq:opid} for any function $f \in L^p(\mathbb{R}^{n-1}, \mathfrak{H}_0)$, we find
\[
\mathcal{R}\mathcal B \mathcal C = \mathcal{R}\left(\frac{1}{2}\left(f + \mathbf{e}_0\mathcal{H}f\right)\right) = f.
\]
Consequently, it follows that $\mathbf{e}_0\mathcal{H}f \in L^p(\mathbb{R}^{n-1}, \eta\mathfrak{H}_0)$.

	And for any \(F \in H^p(\mathbb{R}_{+}^n, \mathcal{R}_n)\) with corresponding boundary function \(f = \mathcal BF \in L^p(\mathbb{R}^{n-1}, \mathcal{R}_n)\), we note that \(f\) decomposes into
	\[f = f_0 + \eta f_1, \quad f_0, f_1 \in L^p(\mathbb{R}^{n-1}, \mathfrak{H}_0).\]
	According to Theorem \ref{thm:Hp}, \(f = \mathbf{e}_0\mathcal{H}f\), leading to the equalities
	\[f_0 + \eta f_1 = \mathbf{e}_0\mathcal{H}f_0 + \mathbf{e}_0\mathcal{H}\eta f_1.\]
	By comparing components, we find
	\[f_0 = \mathbf{e}_0\mathcal{H}\eta f_1, \quad f_1 = -\eta\mathbf{e}_0\mathcal{H}f_0.\]
	Hence, by applying Theorem \ref{thm:Hp} once more, we conclude
\begin{eqnarray*}
 \mathcal C\mathcal{R} \mathcal BF&=& \mathcal C \mathcal B \mathcal C\mathcal{R} \mathcal BF=\mathcal C\mathcal B \mathcal C\mathcal{R}f=\mathcal C\mathcal B(2\mathcal Cf_0)=\mathcal C(f_0+\mathbf{e_0}\mathcal{H}f_0)\\
	&=& \mathcal C(f_0+\eta f_1)=\mathcal C(f)=F
\end{eqnarray*}
	for any $  F\in H^p(\mathbb{R}_{+}^n,\mathcal{R}_n).$

The continuity of the mapping $$ \mathcal{R}\mathcal B: H^p(\mathbb{R}^{n}_+, \mathcal{R}_n) \rightarrow L^p(\mathbb{R}^{n-1}, \mathfrak{H}_0)$$  is established directly through Theorem \ref{thm:bdv}. Specifically,
\[
\|\mathcal BF\|_{L^p(\mathbb{R}^{n-1}, \mathcal{R}_n)} = \lim_{t \rightarrow 0} \left( \int_{\mathbb{R}^{n-1}} |F(t, \underline{x})|^p \, \mathrm{d}x \right)^{\frac{1}{p}} \leqslant \sup_{t > 0} \left( \int_{\mathbb{R}^{n-1}} |F(t, \underline{x})|^p \, \mathrm{d}x \right)^{\frac{1}{p}} = \|F\|_{H^p(\mathbb{R}^{n}_+, \mathcal{R}_n)}.
\]
Hence, it follows that
\[
\|\mathcal{R} \mathcal BF\|_{L^p(\mathbb{R}^{n-1}, \mathfrak{H}_0)} \leqslant 2\|\mathcal BF\|_{L^p(\mathbb{R}^{n-1}, \mathcal{R}_n)} \leqslant 2\|F\|_{H^p(\mathbb{R}^{n}_+, \mathcal{R}_n)}.
\]
This assertion concludes the proof for the scenario where \(\mathfrak{H} = \mathcal{R}_n\).

{\color{red}
Now, for a general Clifford module \(\mathfrak{H}\), the semisimplicity of Clifford algebras implies that \(\mathfrak{H}\) is a direct sum of copies of \(\mathcal{R}_n\) or \(\mathcal{R}_n^+, \mathcal{R}_n^-\). Moreover, for each copy of \(\mathcal{R}_n\) or \(\mathcal{R}_n^+, \mathcal{R}_n^-\) we can choose the same vector \(\eta\) for a fixed $n$ (see Section~7); hence this proof is also valid for any Clifford module \(\mathfrak{H}\).
}

	{\textbf{($\Rightarrow$)}}
Assuming the existence of an element \(\eta \in \mathbb{R}^n\) with \(\eta^2 = -1\) and a subspace \(\mathfrak{H}_0\) within \(\mathfrak{H} := \mathcal{R}_n\) that satisfy Gilbert's conjecture's conditions, let us consider
\[
\mathfrak{H}_0 = \text{span}_{\mathbb{R}}\{\xi_1, \xi_2, \ldots, \xi_m\},
\]
where \(\{\xi_1, \xi_2, \ldots, \xi_m\}\) forms an orthonormal basis of \(\mathfrak{H}_0\).

Let us consider a real-valued Schwartz function \(f \in \mathcal{S}(\mathbb{R}^{n-1})\), selected to satisfy the condition that for each \(j = 1, 2, \ldots, n-1\),
\[
\mathcal{R}_j f(0)   = \delta_{j1}.
\]
Notice that
\[
\mathcal{R}_j f(0) = \frac{2}{\omega_n} \int_{\mathbb{R}^{n-1}} \frac{-u_j}{|u|^n} f(u) \, \mathrm{d}u.
\]
An example of such a function is $$f(x) = c x_1 e^{-|x|^2},$$  where the constant \(c\) is determined to ensure \(\mathcal{R}_1 f(0) = 1\).

Consider the function \(g\) defined by   $$  g = \xi_1 f,$$   where \(f\) is a previously defined real-valued function. It follows that \(g \in L^p(\mathbb{R}^{n-1}, \mathfrak{H}_0)\) for \(p > 1\). Under the assumption that Gilbert's conjecture is valid, we have that
\[
\mathcal{R}\mathcal{B}\mathcal{C}g = g,
\]
indicating that \(g\) remains invariant under the operation of the composition of the operators \(\mathcal{R}\), \(\mathcal{B}\), and \(\mathcal{C}\).
This yields
\begin{equation}\label{eq:everyw}
	\mathcal{R}\left(\frac{1}{2}\left(g + \mathbf{e}_0\mathcal{H}g\right)\right) = g,
\end{equation}
which implies that $$\mathbf{e}_0\mathcal{H}g \in L^p(\mathbb{R}^{n-1}, \eta\mathfrak{H}_0).$$  Specifically,
\begin{equation}
	\mathbf{e}_0\mathcal{H}g(0) =  \mathbf{e}_0\sum_{j=1}^{n-1}\mathbf{e}_j\mathcal R_jg(0) = \mathbf{e}_0\mathbf{e}_1\xi_1 \in \eta\mathfrak{H}_0.
\end{equation}
Thus, it follows that $\eta\mathbf{e}_0\mathbf{e}_1\xi_1 \in \mathfrak{H}_0$. By similar reasoning,
\[
\eta\mathbf{e}_0\mathbf{e}_j\mathfrak{H}_0 \subseteq \mathfrak{H}_0
\]
for each $j = 1, \ldots, n-1$. Given that $\eta\mathbf{e}_0\mathbf{e}_j$ are invertible, we have
\[
\dim_{\mathbb{R}} \eta\mathbf{e}_0\mathbf{e}_j\mathfrak{H}_0 = \dim_{\mathbb{R}} \mathfrak{H}_0,
\]
leading to
\[
\eta\mathbf{e}_0\mathbf{e}_j\mathfrak{H}_0 = \mathfrak{H}_0
\]
for each $j = 1, \ldots, n-1$.

Moreover, for any function \(F \in H^p(\mathbb{R}_{+}^n, \mathcal{R}_n)\), consider the transformation defined by \(f = \mathcal{B}F\), which ensures that
\[
f \in L^p(\mathbb{R}^{n-1}, \mathcal{R}_n).
\]
This formulation implies that the function \(f\) can be expressed as a decomposition of the form
\[
f = f_0 + \eta f_1,
\]
where \(f_0, f_1 \in L^p(\mathbb{R}^{n-1}, \mathfrak{H}_0)\)

Consequently, we obtain the following relationship for \(f\):
\[
f = \mathcal{B}F = \mathcal{B}\mathcal{C}\mathcal{R}\mathcal{B}F = \mathcal{B}\mathcal{C}\mathcal{R}(f_0 + \eta f_1) = \mathcal{B}\mathcal{C}(2f_0) = f_0 + \mathbf{e}_0\mathcal{H}f_0.
\]
This equation establishes the relationship
\[
\eta f_1 = \mathbf{e}_0\mathcal{H}f_0.
\]
Moreover, applying the operator \(\mathbf{e}_0\mathcal{H}\) to \(\eta f_1\) yields
\[
\mathbf{e}_0\mathcal{H}\eta f_1 = f_0 \in L^p(\mathbb{R}^{n-1}, \mathfrak{H}_0).
\]
Following a similar argument, we deduce that
\[
\mathbf{e}_0\mathbf{e}_j\eta\mathfrak{H}_0 = \mathfrak{H}_0
\]
for each \(j = 1, \ldots, n-1\), indicating a systematic relationship between these operators and the functional space \(L^p(\mathbb{R}^{n-1}, \mathfrak{H}_0)\).

This completes the proof.
\end{proof}

\begin{remark}
	This remark elucidates two aspects of our analysis:
	
	\begin{enumerate}
		\item Concerning equation \eqref{eq:everyw}, Theorem \ref{thm:Hp} posits that
		\[
	\mathcal	B( \mathcal C f(z)) = \frac{1}{2}\big(I + \mathbf{e}_0\mathcal{H}\big)f(x)
		\]
		for almost every \(x \in \mathbb{R}^{n-1}\). By applying \eqref{eq:everyw} universally across \(\mathbb{R}^{n-1}\), and specifically at \(x = 0\), we deduce that
		\[
	\mathcal	B( \mathcal C f) (x) = \lim_{z \to x \, n.t.} \frac{1}{2} \Big( P_t \ast (Id + \mathbf{e}_0\mathcal{H})g\Big)(y),
		\]
		where \(z = (t, y) \in \mathbb{R}^n\).

We observe that \(f\) is a Schwartz function, and its Riesz transform \(\mathcal{R}_j f\) can be represented as a convolution with a tempered distribution \(W_j\), defined by
\[
\langle W_j, \varphi\rangle = \frac{2}{w_{n+1}} \lim_{\epsilon \to 0} \int_{\abs{y} \geqslant \epsilon} \frac{y_j}{\abs{y}^{n+1}} \varphi(y) \, \mathrm{d}y,
\]
as described in \cite[P.325]{GTM249}. According to \cite[P.23, Theorem 3.13]{SW71}, \(\mathcal{R}_j f\) is smooth. Furthermore, as outlined in \cite[P.76]{GSM29},
\[
\mathcal{R}_j f = \left(-\frac{\mathrm{i} \xi_j}{\abs{\xi}} \hat{f} \right)^{\vee}(x),
\]
implying that
\[
\|\mathcal{R}_j f\|_{\infty} \leqslant \left\|\frac{\mathrm{i} \xi_j}{\abs{\xi}} \hat{f}\right\|_1 \leqslant \|\hat{f}\|_1.
\]
This demonstrates that \(f\) and \(\mathcal{R}_j f\) are smooth and bounded. Consequently, for functions \(g\) and its Hilbert transform \(\mathcal{H}g\), their Poisson integrals are guaranteed to converge non-tangentially everywhere, as detailed in \cite[P.62, Theorem 3.16]{SW71}.

\medskip

		\item
		Further insights are provided into the structural properties of the space \(\mathcal{R}_n\), detailed as follows:
		\begin{enumerate}
			\item The space \(\mathcal{R}_n\) can be decomposed into the sum \(\mathfrak{H}_0 \oplus \eta\mathfrak{H}_0\), illustrating that the real dimension of \(\mathfrak{H}_0\) is half that of \(\mathcal{R}_n\), i.e.,
			\[
			\dim_{\mathbb{R}}\mathfrak{H}_0 = \frac{1}{2} \dim_{\mathbb{R}}\mathcal{R}_n.
			\]
			
			\item The identities \(\eta\mathbf{e}_0\mathbf{e}_j\mathfrak{H}_0 = \mathfrak{H}_0\) and \(\mathbf{e}_0\mathbf{e}_j\eta\mathfrak{H}_0 = \mathfrak{H}_0\) for each \(j = 1, \ldots, n-1\) introduce a \(C\ell_{n-2}\) module structure on \(\mathfrak{H}_0\). This endows \(\mathfrak{H}_0\) with a dimensionality that satisfies
			\[
			\dim_{\mathbb{R}}\mathfrak{H}_0 \geqslant \dim_{\mathbb{R}}\mathcal{R}_{n-2}.
			\]
		 Further discussion on this module structure will be presented in the following section.
		\end{enumerate}
		\end{enumerate}
		 \end{remark}

\section{The failure of Gilbert's Conjecture for \(n = 8k + 6, 8k + 7\)}
This section establishes the infeasibility of Gilbert's Conjecture for the dimensions \(n = 8k + 6, 8k + 7\).

\begin{theorem}\label{Thm:false}
	Gilbert's Conjecture fails for dimensions \(n = 8k + 6\) and \(n = 8k + 7\), where \(k \in \mathbb{N}\). This conclusion follows from the requirement that for any \(\eta \in \mathbb{R}^n\) and subspace \(\mathfrak{H}_0\) of \(\mathcal{R}_n\) satisfying the conjecture's criteria, the following inequality must hold:
	\[\dim_{\mathbb{R}}\mathcal{R}_n \geqslant  2  \dim_{\mathbb{R}}\mathcal{R}_{n-2}.\]
\end{theorem}

\begin{proof}
	The assertion that
	\[\eta \mathbf{e}_0 \mathbf{e}_j \mathfrak{H}_0 = \mathfrak{H}_0 \quad \text{and} \quad \mathbf{e}_0 \mathbf{e}_j \eta \mathfrak{H}_0 = \mathfrak{H}_0, \quad j = 1, \ldots, n-1,\]
	induces a \(C\ell_{n-2}\) module structure on \(\mathfrak{H}_0\). Define \(g_j = \mathbf{e}_0 \mathbf{e}_j, j = 1, \ldots, n-1\), making these conditions equivalent to
	\begin{equation}\label{eq:g}
		g_j \mathfrak{H}_0 = \eta \mathfrak{H}_0 \quad \text{and} \quad g_j \eta \mathfrak{H}_0 = \mathfrak{H}_0, \quad j = 1, \ldots, n-1.
	\end{equation}	
We assert that the algebra generated by \(\{g_j = \mathbf{e}_0 \mathbf{e}_j, j = 1, \ldots, n-1\}\) is isomorphic to \(C\ell^{even} \cong C\ell_{n-1}\). This follows from standard Clifford algebra results, but a brief demonstration is given for completeness.
	
	Consider the embedding \(v: \mathbb{R}^{n-1} \rightarrow \mathbb{A}\) defined by \(v((x_1, \ldots, x_{n-1})) = \sum_{j=1}^{n-1} x_j g_j\), yielding
	\begin{equation*}
		v((x_1, \ldots, x_{n-1}))^2 = -\lVert x \rVert^2.
	\end{equation*}
	When \(n\) is odd (\(n-1\) is even), the induced homomorphism from the Clifford algebra is an isomorphism. When \(n\) is even (\(n-1\) is odd), \(g_1 \cdots g_{n-1} \notin \mathbb{R}\), indicating \(\mathbb{A} \cong C\ell_{n-1}\).

	Furthermore, following condition \eqref{eq:g} $$g_ig_j\mathfrak{H}_0=g_i\eta\mathfrak{H}_0=\mathfrak{H}_0, $$ we observe that \(\mathfrak{H}_0\) is a module over the algebra generated by the set \(\{g_1g_j, j=2, \ldots, n-1\}\). This algebra is notably isomorphic to \(C\ell_{n-2}\).
	
	Hence, the real dimension of \(\mathfrak{H}_0\) satisfies the following relations:
	\begin{align*}
		\dim_{\mathbb{R}}\mathfrak{H}_0 &\geqslant \dim_{\mathbb{R}}\mathcal{R}_{n-2}, \\
		\dim_{\mathbb{R}}\mathfrak{H}_0 &= \frac{1}{2} \dim_{\mathbb{R}}\mathcal{R}_n.
	\end{align*}
	
	These relationships culminate in a critical dimensionality condition:
	\begin{equation}\label{eq:dimineq}
		\dim_{\mathbb{R}}\mathcal{R}_n \geqslant 2 \dim_{\mathbb{R}}\mathcal{R}_{n-2}.
	\end{equation}
	
	By invoking Remark \ref{thm:CLs}, we deduce  that the dimensions of \(\mathcal{R}_n\) exhibit a specific pattern:
	\begin{equation*}
		\dim_{\mathbb{R}}\mathcal{R}_{8k+4} = \dim_{\mathbb{R}}\mathcal{R}_{8k+5} = \dim_{\mathbb{R}}\mathcal{R}_{8k+6} = \dim_{\mathbb{R}}\mathcal{R}_{8k+7^{\pm}} = 128k.
	\end{equation*}
	
	This observation leads to the conclusion that the dimensionality condition expressed in inequality 	\eqref{eq:dimineq} fails to hold for these specific cases. Consequently, Gilbert's conjecture is disproven for these instances.
	\end{proof}

\section{Validity of Gilbert's Conjecture for   \(n \neq 8k+6\) or \(8k+7\) }

In this section, we explore the scenarios in which Gilbert's conjecture is upheld, except for the specified forms where \(n\) is an integer multiple of 8 plus 6 or 7.

\begin{theorem}\label{Thm:true}
	Gilbert's conjecture is affirmed for cases where \(n \neq 8k+6\) and \(8k+7\), with \(k\) being any natural number.
\end{theorem}

\noindent {\bf Convention:} \ \ In the proof of Theorem \ref{Thm:true}, a notational convention is adopted where $$\bfe_i^{(j)} \in V_1 \oplus V_2 \oplus \cdots \oplus V_n$$  represents the \((i+1)\)-th canonical orthonormal  basis vector of \(V_j\) for every \(i=0, 1, \cdots, \dim_{\mathbb R} V_j-1 \) and  \(j=1, \cdots, n\).

\begin{proof}[Proof of Theorem \ref{Thm:true}] We will give a specific choice of $\eta$ and $\mathfrak{H}_0\subset \mathcal{R}_n$ by   Remark  \ref{thm:CLs}.
\begin{enumerate}
  \item When \(n = 8k\) for \(k \in \mathbb{N}_+\), take \(\eta = \mathbf{e}^{(1)}_0 = (1,0,\ldots,0) \in \mathbb{O}^{k} = \mathbb{R}^{8k}\), and let
  \[
  \mathfrak{H}_0 := \mathbb{O}\begin{pmatrix} 1 \\ 1 \end{pmatrix} \otimes (\mathbb{O}^2)^{\otimes (k-1)}.
  \]

 Given the definition of \(v_{8k}\) as specified in Table \ref{Tab:2}, we derive the following relation:
 \begin{equation*}
 	\eta \mathfrak{H}_0 = \left\{ \left( \begin{matrix} 0 & 1 \\ -1 & 0 \end{matrix} \right) \otimes E_{k-1} \left( \begin{matrix} p \\ p \end{matrix} \right) \otimes (\mathbb{O}^2)^{\otimes (k-1)} : p \in \mathbb{O} \right\} = \mathbb{O} \begin{pmatrix} 1 \\ -1 \end{pmatrix} \otimes (\mathbb{O}^2)^{\otimes (k-1)},
 \end{equation*}
 leading to the space \(\mathcal{R}_{8k}\) being represented as
 \begin{equation*}
 	\mathcal{R}_{8k} = \mathfrak{H}_0 \oplus \eta \mathfrak{H}_0.
 \end{equation*}

 According to Theorem \ref{thm:redu-134}, to establish the desired property, it is sufficient to demonstrate that \(\mathbf{e}_j \mathfrak{H}_0 = \mathfrak{H}_0\) for \(j = 1, \ldots, n-1\). Based on \(v_{8k}\)'s definition, it naturally follows that \(\mathbf{e}^{(j)}_i \mathfrak{H}_0 = \mathfrak{H}_0\) for \(i = 0, \ldots, 7\), and \(j = 2, \ldots, k\).

 For illustrative purposes, we detail the case of \(\mathbf{e}^{(2)}_i \mathfrak{H}_0 = \mathfrak{H}_0\) as an example:
 \begin{equation*}
 \hspace{35pt}	v_{8k}(\mathbf{e}_j^{(2)}) \mathfrak{H}_0 = \left[ \text{Id}_{\mathbb{O}^2}\otimes\left( \begin{matrix} 0 & R_{\mathbf{e}_j} \\ -R_{\bar{\mathbf{e}}_j} & 0 \end{matrix} \right) \otimes E_{k-2} \right] \left( \begin{matrix} p \\ p \end{matrix} \right) \otimes (\mathbb{O}^2)^{\otimes (k-1)} = \left( \begin{matrix} p \\ p \end{matrix} \right) \otimes (\mathbb{O}^2)^{\otimes (k-1)}.
 \end{equation*}

   Furthermore,
  \[
  \qquad\qquad  v_{8k}(\mathbf{e}_j^{(1)}) \mathfrak{H}_0 = \left[ \left( \begin{matrix} 0 & R_{\mathbf{e}_j} \\ R_{\mathbf{e}_j} & 0 \end{matrix} \right) \otimes E_{k-1} \right] \left( \begin{matrix} p \\ p \end{matrix} \right) \otimes (\mathbb{O}^2)^{\otimes (k-1)}
  = \left( \begin{matrix} p \mathbf{e}_j \\ p \mathbf{e}_j \end{matrix} \right) \otimes (\mathbb{O}^2)^{\otimes (k-1)}.
  \]
  Hence, the condition is satisfied in this case.

\item For the case when \(n = 8k + 1\) with \(k \in \mathbb{N}_+\), we define \(\eta\) as \(\mathbf{e}_0^{(1)} = (1,0,\ldots,0)\), which lies in the direct sum \(\mathbb{R} \oplus \mathbb{O}^{k}\). Consider the construction of \(\mathfrak{H}_0\) as follows:
\[
\mathfrak{H}_0 := 1 \otimes (\mathbb{O}^2)^{\otimes k}.
\]
Referring to Table \ref{Tab:2} for the definition of \(v_{8k+1}\), we obtain the transformation of \(\eta\) applied to \(\mathfrak{H}_0\):
\[
v_{8k+1}(\eta) \mathfrak{H}_0 = \left[ L_{\mathrm{i}} \otimes E_k \right] 1 \otimes (\mathbb{O}^2)^{\otimes k} = \mathrm{i} \otimes (\mathbb{O}^2)^{\otimes k},
\]
which leads to the conclusion that:
\[
\mathcal{R}_{8k+1} = \mathfrak{H}_0 \oplus \eta \mathfrak{H}_0.
\]
To verify that \(\mathbf{e}^{(j)}_i \mathfrak{H}_0 = \mathfrak{H}_0\) for all \(i = 0, \ldots, 7\) and \(j = 2, \ldots, k + 1\), we follow a computation analogous to the one described in Case  (1). This verification process affirms the validity of the structure for \(\mathcal{R}_{8k+1}\) and completes the proof for this particular case.

\item For the case where \(n = 8k + 2\) with \(k \in \mathbb{N}\), we define \(\eta\) as \(\mathbf{e}_0^{(1)} := ((1,0), 0, \ldots, 0)\), which resides in the space \(\mathbb{R}^2 \oplus \mathbb{O}^{k}\). Consider the module \(\mathfrak{H}_0\) defined by:
\[
\mathfrak{H}_0 = \mathbb{C}_{\mathrm{j}} \otimes (\mathbb{O}^2)^{\otimes k},
\]
where \(\mathbb{C}_{\mathrm{j}} = \{a + b\mathrm{j} \in \mathbb{H} \,|\, a,b \in \mathbb{R}\}\) denotes the \(j\)-complex plane within the quaternions \(\mathbb{H}\).

Referring to the definition of \(v_{8k+2}\) as outlined in Table \ref{Tab:2}, we find that the application of \(\eta\) to \(\mathfrak{H}_0\) yields:
\[
v_{8k+2}(\eta) \mathfrak{H}_0 = \left[ L_{\mathrm{i}} \otimes E_k \right] \mathbb{C}_{\mathrm{j}} \otimes (\mathbb{O}^2)^{\otimes k} = \mathrm{i} \mathbb{C}_{\mathrm{j}} \otimes (\mathbb{O}^2)^{\otimes k},
\]
which leads to the representation of \(\mathcal{R}_{8k+2}\) as:
\[
\mathcal{R}_{8k+2} = \mathfrak{H}_0 \oplus \eta \mathfrak{H}_0.
\]
Additionally, the transformation induced by \(v_{8k+2}(\mathbf{e}_1^{(1)})\) on \(\mathfrak{H}_0\) confirms that:
\[
v_{8k+2}(\mathbf{e}_1^{(1)}) \mathfrak{H}_0 = \left[ L_{\mathrm{j}} \otimes E_k \right] \mathbb{C}_{\mathrm{j}} \otimes (\mathbb{O}^2)^{\otimes k} = \mathfrak{H}_0.
\]
For any other vector \(\mathbf{e}^{(j)}_i\), with \(i = 0, \ldots, 7\) and \(j = 2, \ldots, k+1\), it is observed that these vectors act trivially on the \(\mathbb{C}_{\mathrm{j}}\) component, thereby preserving the decomposition of \(\mathfrak{H}_0\).

This analysis completes the proof for the case \(n = 8k + 2\), establishing the structured decomposition of \(\mathcal{R}_{8k+2}\).

{\color{red} \item When \(n = 8k + 3, 8k+5\) with \(k \in \mathbb{N}\), 
Since $\eta=\mathbf{e}_0=\mathbf{e}_0^{(1)}$is not possible in these cases (see the last section), we treat them separately.

For \(k \in \mathbb{N}_+\), we set $W_k=(\mathbb{O}^2)^{\otimes k}$.
Since $E_k^2=Id_{W_k}$, define
$$W_k^+=\{x\in W_k: E_kx=x\}, \ W_k^-=\{x\in W_k: E_kx=-x\},$$
then we have $W_k=W_k^+\oplus W_k^-$.
By Lemma 3.2, for every nonzero $p\in \mathbb{O}$ and $s=1, 2, \cdots, k$ we have 	
		\[
		A_p^s W_k^+ = W_k^-, \ \  A_p^s W_k^- = W_k^+.
		\]

\begin{enumerate}
  \item Now consider the scenario where \(n = 8k + 3\), with \(k \in \mathbb{N}_+\). Our focus will be on proving properties related to the space \(\mathcal{R}_{8k+3}^+\), defined as \(\mathcal{R}_{8k+3}^+ = \mathbb{H}^+ \otimes (\mathbb{O}^2)^{\otimes k}\).

We set $A=Span_{\R}\{1,\mathrm{i}\}$ , $B=Span_{\R}\{\mathrm{j},\mathrm{k}\}$ ,
\[
  \mathfrak{H}_0 :=A\otimes W_k^+\oplus B\otimes W_k^-,
\]
and \(\eta = \mathbf{e}_1^{(1)} = (\mathrm{j}, 0, \ldots, 0)\) as an element in the space \(\text{Im}\mathbb{H} \oplus \mathbb{O}^k\).
Upon analyzing the effect of the operator \(v_{8k+3}\) on the element \(\eta\), as detailed in Table \ref{Tab:2}, $$v_{8k+3}(\eta)=R_{\mathrm{j}} \otimes E_k,$$
so we have $\eta\mathfrak{H}_0=\left(R_{\mathrm{j}} \otimes E_k\right)\left(A\otimes W_k^+\oplus B\otimes W_k^-\right)=B\otimes W_k^+\oplus A\otimes W_k^-$.
This leads us to conclude that \(\mathcal{R}_{8k+3}^+ = \mathfrak{H}_0 \oplus \eta\mathfrak{H}_0\).
Further observations reveal that
  \[
  \eta\mathbf{e}^{(1)}_0\mathbf{e}^{(1)}_1 = \mathbf{e}_0 = \left(R_{\mathrm{i}} \otimes E_k\right), \quad \eta\mathbf{e}^{(1)}_0\mathbf{e}^{(1)}_2 = -1 \otimes E_k.
  \]
  Moreover, for any \(\mathbf{e}^{(j)}_i\), with \(i = 0, \ldots, 7\), and \(j = 2, \ldots, k+1\), we have
   $$\eta\mathbf{e}^{(1)}_0\mathbf{e}^{(j)}_i=R_{\mathrm{k}}\otimes A_p^s,$$ for some $p\in \mathbb{O}$.
  It is straightforward to verify that \(\eta\mathbf{e}^{(1)}_0\mathbf{e}^{(j)}_i\mathfrak{H}_0 = \mathfrak{H}_0\).

  In the special instance when \(n=3\), the previously defined \(\mathfrak{H}_0\) does not hold. For such a case, selecting \(\eta = \mathrm{j}\) and setting \(\mathfrak{H}_0\) to \(\mathbb{C}_{\mathrm{i}}\) are deemed appropriate, adapting the structure to suit the unique properties of this scenario.
  \item  Consider the case when \(n = 8k + 5\) with \(k \in \mathbb{N}\), $\mathcal{R}_{8k+5}=\mathbb{H}^2\otimes\left(\mathbb{O}^2\right))^{\otimes k}$. Let us set \(\eta = (0, 1, 0, \ldots, 0) = \mathbf{e}_0^{(2)}\) as an element within \(\mathbb{R} \oplus \mathbb{H} \oplus \mathbb{O}^k\). The subspace \(\mathfrak{H}_0\) is then defined as
  \[
   \mathfrak{H}_0 :=\begin{pmatrix} \mathbb{H} \\ 0 \end{pmatrix}\otimes W_k^+\oplus \begin{pmatrix} 0 \\ \mathbb{H} \end{pmatrix}\otimes W_k^-,
  \]
  Following the definition of \(v_{8k+5}\), the action on \(\eta\) yields
  \begin{eqnarray*}
    v_{8k+5}(\eta)\mathfrak{H}_0 :&=& \left[\left(
		\begin{array}{cc}
			0 & Id_{\mathbb{H}} \\
			-Id_{\mathbb{H}} & 0 \\
		\end{array}
		\right) \otimes E_k \right]
\left[\begin{pmatrix} \mathbb{H} \\ 0 \end{pmatrix}\otimes W_k^+\oplus \begin{pmatrix} 0 \\ \mathbb{H} \end{pmatrix}\otimes W_k^-\right] \\
       &=& \begin{pmatrix}  0 \\ \mathbb{H} \end{pmatrix}\otimes W_k^+\oplus \begin{pmatrix}  \mathbb{H} \\ 0 \end{pmatrix}\otimes W_k^-
    \end{eqnarray*}

  This implies that
  \[
  \mathcal{R}_{8k+5} = \mathfrak{H}_0 \oplus \eta\mathfrak{H}_0.
  \]
  Furthermore, direct computation reveals that
  \[
  \eta\mathbf{e}_0^{(1)}\mathbf{e}_0^{(2)} = \mathbf{e}_0^{(1)} = \begin{pmatrix} L_{\mathrm{i}} & 0 \\ 0 & -L_{\mathrm{i}} \end{pmatrix} \otimes E_k
  \]
  and for \(j = 1, \ldots, 3\),
  \[
  \eta\mathbf{e}_0^{(1)}\mathbf{e}_j^{(2)} = \begin{pmatrix} -L_{\mathrm{i}}R_{\mathbf{e}_j} & 0 \\ 0 & -L_{\mathrm{i}}R_{\mathbf{e}_j} \end{pmatrix} \otimes E_k,
  \]

  Moreover, for any \(\mathbf{e}^{(j)}_i\), with \(i = 0, \ldots, 7\), and \(j = 3, \ldots, k+1\), we have
   $$\eta\mathbf{e}^{(1)}_0\mathbf{e}^{(j)}_i=\begin{pmatrix} 0 & -L_{\mathrm{i}} \\ -L_{\mathrm{i}} & 0 \end{pmatrix}\otimes A_p^s,$$ for some $p\in \mathbb{O}$. It is straightforward to verify that \(\eta\mathbf{e}^{(1)}_0\mathbf{e}^{(j)}_i\mathfrak{H}_0 = \mathfrak{H}_0\).

In the special case where \(n=5\), the previously established definition of \(\mathfrak{H}_0\) is not suitable. For this particular scenario, selecting \(\eta\) as \(\mathbf{e}_0^{(2)}\) and redefining \(\mathfrak{H}_0\) to be
\[
\mathfrak{H}_0 = \begin{pmatrix} \mathbb{H} \\ 0 \end{pmatrix}
\]
proves to be the appropriate adaptation. This choice aligns with the unique structural requirements and mathematical framework necessitated by the specific case of \(n=5\).

\end{enumerate}

}

  \item When $n=8k+4,k\in\mathbb{N}$,
  Set $\eta=\bfe_0^{(1)}=(1,0,\cdots,0)\in \mathbb{H}\oplus \mathbb{O}^k$,
  \begin{eqnarray*}
    \mathfrak{H}_0:=\mathbb{H}\left(
                                                                                                   \begin{array}{c}
                                                                                                     1 \\
                                                                                                     1 \\
                                                                                                   \end{array}
                                                                                                 \right)\otimes\left(\mathbb{O}^2\right)^{\otimes k}.
  \end{eqnarray*}
 By the definition of $v_{8k+4}$,
 \begin{eqnarray*}
v_{8k+4}(\eta)\mathfrak{H}_0 & = &\left[\left(\begin{matrix}
		0&1\\
		-1&0
	\end{matrix}\right)\otimes E_{k}\right]\left(\left(\begin{matrix}
		p\\
		p
	\end{matrix}\right)\otimes\left(\mathbb{O}^2\right)^{\otimes k}\right)\\
&=&\left(\begin{matrix}
		p\\
		-p
	\end{matrix}\right)\otimes\left(\mathbb{O}^2\right)^{\otimes k}=\mathbb{H}\left(\begin{matrix}
		1\\
		-1
	\end{matrix}\right)\otimes\left(\mathbb{O}^2\right)^{\otimes k}.
\end{eqnarray*}
 So we have $$\mathcal{R}_{8k+4}=\mathfrak{H}_0\oplus\eta\mathfrak{H}_0.$$
 Performing a computation analogous to that in Case (1) for each \(j=1,\ldots,3\), we evaluate the action of \(v_{8k+4}\) on \(\mathfrak{H}_0\) via \(\mathbf{e}_j^{(1)}\):
 \begin{equation*}
 	v_{8k+4}(\mathbf{e}_j^{(1)})\mathfrak{H}_0 = \left[\left(\begin{matrix}
 		0 & R_{\mathbf{e}_j} \\
 		R_{\mathbf{e}_j} & 0
 	\end{matrix}\right) \otimes E_{k}\right]\left(\begin{matrix}
 		p \\
 		p
 	\end{matrix}\right) \otimes \left(\mathbb{O}^2\right)^{\otimes k}
 	= \left(\begin{matrix}
 		p\mathbf{e}_j \\
 		p\mathbf{e}_j
 	\end{matrix}\right) \otimes \left(\mathbb{O}^2\right)^{\otimes k}.
 \end{equation*}
 This result completes the analysis for the specified case.

\end{enumerate}
This finishes the proof.

\end{proof}

\section{ Remarks about specific selection of the vector  $\eta$}

In the course of proving Theorem \ref{Thm:false}, it was established that $\mathfrak{H}_0$ must exhibit a $C\ell_{n-2}$ module structure. However, should we select $\eta = \mathbf{e}_0$, and the condition
\[
\mathbf{e}_j\mathfrak{H}_0 = \mathfrak{H}_0, \quad j = 1, \ldots, n-1,
\]
holds, then it necessitates that $\mathfrak{H}_0$ adopts a $C\ell_{n-1}$ module structure. This implication, when juxtaposed with the dimensional properties of the spinor spaces $\mathcal{R}_n$, suggests that
\[
\dim_{\mathbb{R}}\mathcal{R}_n \geqslant  2\dim_{\mathbb{R}}\mathcal{R}_{n-1}
\]
in these instances. More specifically,

\begin{theorem}\label{thm:noe0}
	When $n = 8k + 3$ or $8k + 5$, for any $k \in \mathbb{N}$, regardless of how we permute the order of $\mathbf{e}_j, j = 0, \ldots, n-1$, the vector $\eta$ cannot be chosen as $\mathbf{e}_0$.
\end{theorem}

Nevertheless, for cases where $n = 8k + 3$ or $8k + 5$, $\mathfrak{H}_0$ is mandated to possess a $C\ell_{1,n-2}$ module structure. Here, $C\ell_{p,q}$ represents the Clifford algebra over the quadratic space $(\mathbb{R}^{p+q}, Q)$, defined by
\[
Q(x) = \sum_{i=1}^{p} x_i^2 - \sum_{j=1}^{q} x_{p+j}^2.
\]

It is noteworthy that the algebra generated by $\{\mathbf{e}_1\mathbf{e}_0\mathbf{e}_j, \quad j=1, \ldots, n-1\}$ equates to $C\ell_{1,n-2}$. Furthermore,
\[
C\ell_{1,1} \cong M(2, \mathbb{R}), \quad C\ell_{1,3} \cong M(4, \mathbb{R})
\]
(as detailed in \cite{LM89}), thereby implying that
\[
\dim_{\mathbb{R}}\mathcal{R}_{1,1} = 2=\frac{1}{2}\dim_{\mathbb{R}}\mathcal{R}_3, \quad \dim_{\mathbb{R}}\mathcal{R}_{1,3} = 4=\frac{1}{2}\dim_{\mathbb{R}}\mathcal{R}_5.
\]
This underpins the feasibility of selecting such an $\eta = \mathbf{e}_1$ and establishing the existence of $\mathfrak{H}_0$.

In conclusion, this paper highlights the variability in choosing $\eta$ and $\mathfrak{H}_0$. While proving Theorem \ref{Thm:true}, a particular selection was presented. For instance, in the scenario of $n = 8k + 5$, alternatives such as $\eta = \mathbf{e}_1^{(2)}$ or any basis vector other than \(\mathbf{e}_0\) could be a viable choice.

\section*{Conflict of interest statement}
We declare that we have no conflict of interest.

\bigskip\bigskip

\end{document}